\documentclass[11pt]{article}

\usepackage{amsmath}
\usepackage{amsfonts}
\usepackage{amsthm}
\usepackage{graphicx}
\usepackage{caption}
\usepackage{subcaption}
\usepackage{setspace}
\usepackage{fullpage}
\usepackage{hyperref}
\usepackage{epstopdf}
\usepackage{cleveref} 
\usepackage{wrapfig}
\usepackage{float}
\usepackage{multicol}

\theoremstyle{definition}
\newtheorem{theorem}{Theorem}[section]
\newtheorem{remark}[theorem]{Remark}
\newtheorem{proposition}[theorem]{Proposition}
\newtheorem{definition}[theorem]{Definition}
\newtheorem{lemma}[theorem]{Lemma}
\newtheorem{corollary}[theorem]{Corollary}

\providecommand{\keywords}[1]{\textbf{\textit{Keywords:}} #1}

\DeclareMathOperator*{\argmin}{arg\,min}
\DeclareMathOperator*{\supp}{supp}

\DeclareMathOperator{\sm3}{Sym_0(3)}
\DeclareMathOperator{\ps2}{\mathcal{P}(\mathbb{S}^2)}

\DeclareMathOperator{\pus2}{\mathcal{P}_U(\mathbb{S}^2)}
\DeclareMathOperator{\mf}{\mathcal{F}}
\DeclareMathOperator{\s2}{\mathbb{S}^2}
\DeclareMathOperator{\wsl}{\overset{*}{\rightharpoonup}}

\newcommand{\mesr}{%
  \,\raisebox{-.127ex}{\reflectbox{\rotatebox[origin=br]{-90}{$\lnot$}}}\,%
}

\begin{document}
\doublespace 
\title{An analysis of equilibria in dense nematic liquid crystals}
\author{Jamie M. Taylor\footnote{Address for correspondence: Mathematical Institute, Radcliffe Observatory Quarter, Woodstock Road, Oxford, OX2 6GG. Email: \texttt{jamie.taylor@maths.ox.ac.uk}}}
\date{}
\maketitle
\begin{abstract}
This paper is concerned with the rigorous analysis of a recently proposed model of Zheng {\it et. al.} for describing nematic liquid crystals within the dense regime, with the orientation distribution function as the variable. A key feature of the model is that in high density regimes all non-trivial minimisers are zero on a set of positive measure so that $L^\infty$ variations cannot generally be taken about minimisers. In particular, it is unclear if the Euler-Lagrange equation is well defined, and if local minimisers satisfy it. It will be shown that there exists an analogue of the Euler-Lagrange equation that is satisfied by $L^p$ local minimisers by reducing the minimisation problem to an equivalent finite-dimensional saddle-point problem, obtained by observing that on certain subsets of the domain the free-energy functional is convex so that duality methods can be applied. This analogue of the Euler-Lagrange equation is then shown to be equivalent to a vanishing variation criteria on a certain family of non-linear curves on which the free-energy functional is sufficiently smooth. All critical points of the finite-dimensional saddle-point problem also correspond to all probability distributions where these non-linear variations vanish. Furthermore, the analysis provides results on some qualitative phase behaviour of the model. 
\end{abstract}
\section{Introduction}
The Onsager model \cite{onsager1949effects}, based on the second virial expansion, is the now classical variational description of phase transitions in liquid crystalline systems, describing equilibrium configurations of molecular systems by critical points of a free-energy functional. In the simplest case of uniaxial molecules forming a nematic phase with a Maier-Saupe-like excluded volume term \cite{maier1959einfache}, a spatially homogeneous system is identified with a probability distribution on the sphere, $f \in \ps2$, describing the orientations of axially symmetric molecules. At fixed temperature and concentration, we look for local minimisers of 
\begin{equation}\int_{\s2}f(p)\ln f(p)-\frac{U_0\rho f(p)}{2}\int_{\s2}f(q)\left((p\cdot q)^2-\frac{1}{3}\right)\,dq\,dp,\end{equation}
with $U_0>0$ a fixed constant, $\rho>0$ the number density. While the model provides a qualitatively accurate description of the phase behaviour of nematic liquid crystals, the derivation is only valid in more dilute regimes, and in particular there is no barrier to prevent arbitrarily high number density. 

The recent work of Zheng {\it et. al.} \cite{zheng2016density} proposes a new free-energy functional that aims to demonstrate the consequences of a lack of available configuration space in high density regimes. At fixed concentration and in the absence of thermal effects this gives a free energy of the form 
\begin{equation}\mf(f,\eta)=\int_{\s2}f(p)\ln f(p)-f(p)\ln\left(\int_{\s2}\left((p\cdot q)^2-\frac{1}{3}\right)f(q)\,dq-\eta\right)\,dp,\end{equation}
where $\eta=\frac{2(\rho c-1)}{3\rho d}$ is a dimensionless parameter, increasing in the number density $\rho$, and $c,d>0$ are constants related to the dimensions of the molecule. For the majority of this work, $\eta$ will be fixed and the explicit dependence of $\mf$ on $\eta$ will be supressed. The key feature of the model is that at higher densities, molecules must be more strongly aligned in order for the energy to be finite, and minimisers must in some cases be zero on some non-empty subset of $\s2$, in stark contrast to the solutions of the Maier-Saupe model which are always bounded away from zero \cite{fatkullin2005critical}. The interpretation of this is that at higher densities it becomes not just energetically unfavourable but impossible for molecules to align against the order. Furthermore there exists a saturation density at which there are no finite energy configurations.

Within the work of Zheng {\it et. al.} the model was derived from more elementary principles, the Euler-Lagrange equation for minimisers was given and there was a numerical study illustrating novel phase behaviour. This paper aims to rigorously address issues surrounding equilibria raised in their work. In particular, it is not immediately clear if minimisers will satisfy the Euler-Lagrange equation due to having non-trivial support, and generally the free energy can lack sufficient smoothness at local minimisers for arbitrary variations to be taken. In order to establish an Euler-Lagrange equation satisfied by minimisers we will instead split the problem into two more manageable steps, in a similar method to \cite[Section 4.1]{taylor2015maximum}. By restricting ourselves only to probability distributions $f$ such that the so-called Q-tensor 
\begin{equation}Q=\int_{\s2}\left(p\otimes p-\frac{1}{3}I\right)f(p)\,dp\end{equation}
is fixed, we see that the free energy becomes convex with linear constraints and can be tackled using the results of Borwein and Lewis \cite{borwein1991duality}. Once this problem has been tackled, it remains only to minimise over the set of admissible Q-tensors. By considering this finite-dimensional problem, we can obtain the Euler-Lagrange equation, which is consistent with the results of Zheng {\it et. al.}, and prove that $L^p$-local minimisers satisfy it. This is then seen to be equivalent to considering a particular set of curves $f:(-\delta,\delta)\to\ps2$ and solving a vanishing derivative condition $\left.\frac{d}{dt}\mf(f_t)\right|_{t=0}=0$

The structure of the paper and key results are as follows. In \Cref{secMinimisers} we will first prove that global minimisers of the free energy exist if and only if $\eta<\frac{2}{3}$, with no finite energy configurations otherwise (\Cref{theoremExistence}). We will also investigate local minimisers that are bounded away from zero and infinity by taking $L^\infty(\s2)$ variations. It will be shown that such a method can produce certain trivial minimisers, but in general it does not produce satisfactory results, which is more precisely stated in \Cref{corollaryIsotropicDrama}. In particular, only the isotropic state when $\eta<0$ and a continuum of solutions when $\eta=-\frac{2}{15}$ are found by this method. \Cref{secAuxiliaryProblem} will be concerned with reducing the global minimisation problem to a macroscopic minimisation problem over the set of Q-tensors. More precisely, if $A(Q)=\left\{ f \in \ps2 : \int_{\s2}f(p)\left(p\otimes p-\frac{1}{3}I\right)\,dp\right\}$, then we split the minimisation problem as 
\begin{equation}\min\limits_{f \in \ps2}\mathcal{F}(f)=\min\limits_{Q \in \mathcal{Q}} \left(\min\limits_{ f \in A(Q)} \mathcal{F}(f)\right)=\min\limits_{Q \in \mathcal{Q}}J(Q),\end{equation}
where the inner minimisation problem defining the macroscopic functional $J$ is convex with continuous linear constraints, and as such can be tackled by the techniques in \cite{borwein1991duality}. By showing that the macroscopic functional $J$ is sufficiently regular (\Cref{propJDerivatives}),  we can obtain a critical point condition for global minimisers which acts as the Euler-Lagrange equation for the free energy $\mathcal{F}$. \Cref{theoremGlobalMinsEquivalent} is the main result of this section, showing the equivalence of the minimisation problems, an equivalent finite-dimensional saddle-point problem and the relationship between their solutions. Explicitly, global minimisers must be of the form 
\begin{equation}
f(p)=\frac{1}{Z}\exp(\Lambda p\cdot p)\max(Qp\cdot p-\eta,0),
\end{equation}
where $\Lambda,Q \in \sm3$ satisfy 
\begin{equation}
\begin{split}
Q=& \int_{\mathbb{S}^2}f(p)\left(p\otimes p-\frac{1}{3}I\right)\,dp,\\
\Lambda =& \int_{\mathbb{S}^2}\frac{f(p)}{Qp\cdot p-\eta}\left(p\otimes p-\frac{1}{3}I\right)\,dp,
\end{split}
\end{equation}
and also solve the saddle-point problem 
\begin{equation}\min\limits_{Q_0 \in \mathcal{Q}}\max\limits_{\lambda \in \sm3} Q_0 \cdot \lambda - \ln \left(\int_{\s2}\exp(\lambda p\cdot p)\max(Q_0 p\cdot p-\eta,0)\,dp\right).\end{equation}
Furthermore, $Q$ will also be a global minimiser of the macroscopic function $J$. 

In \Cref{secLocalMinimisers}, we will consider local minimisers of $\mathcal{F}$. While the decomposition of the minimisation problem was effective for finding global minimisers, the continuity of the map from a given Q-tensor to its optimal energy probability distribution will also allow us to make similar claims for local minimisers. In particular, local minimisers of the macroscopic function $J$ correspond to local minimisers of $\mathcal{F}$ in a statement highly analogous to that for global minimisers (\Cref{theoremLocalMinsEQuiv}), and they satisfy the same critical point condition. Furthermore the equivalence shows that local minimisers with respect to the topologies of $L^p$ for $p \in [1,+\infty]$, $W^{1,p}$ for $p \in [1,\infty)$ and $C^\alpha$ for $\alpha<1$ are all equivalent (\Cref{corollaryNoLavrentiev}).

\Cref{secRelatedModels} will be concerned with related models that can be tackled using previous results in this work. \Cref{secThermal} will account for thermal effects represented by an adjusted free energy 
\begin{equation}\mathcal{F}_\tau(f)=\mathcal{F}(f)-\frac{1}{2\tau}\left|\int_{\s2}f(p)\left(p\otimes p-\frac{1}{3}I\right)\,dp\right|^2,\end{equation}
where $\tau>0$ is proportional to temperature. The results from previous sections can be extended in a straightforward manner to provide an Euler-Lagrange equation for the thermal model (\Cref{theoremMinsEquivThermal}). A local analysis around the isotropic state for $\eta<0$ demonstrates the stability of the isotropic state for $\tau>\tau_c=\frac{15}{2}\left(\frac{15\eta}{2+15\eta}\right)^{2}$ and instability if $\tau<\tau_c$ (\Cref{propBifurcation}). In particular, this gives a re-emergence of local stability for the isotropic phase as concentration is increased at fixed temperature, in contrast to more classical theories. In \Cref{subsecUniaxial} we consider uniaxial systems as a restricted class of admissible probability distributions that are rotationally invariant about some axis. Experimentally, one often observes such symmetry in nematic systems, and the theory is generally simpler due to having fewer degrees of freedom. In this subsection we rigorously obtain the Euler-Lagrange equation for local minimisers used in \cite{zheng2016density} (\Cref{propMinsEquivUniaxial}). Furthermore, the results in this subsection allows us to show the existence of certain uniaxial critical points to the unconstrained problem in higher density regimes (\Cref{corollaryUniaxial1,corollaryUniaxial2}), providing further qualitative information about the phase diagram for the unconstrained model.

\section{Finding solutions by variations}
\label{secMinimisers}
\begin{definition}[Notation and energy]
Let $\sm3\subset\mathbb{R}^{3\times 3}$ denote the set of traceless, symmetric $3\times 3$ matrices. Let $\mathcal{Q}=\{Q \in \sm3 : v_{\min}(Q)>-\frac{1}{3}\}$ be the set of physical Q-tensors. Let $\mathcal{P}(\mathbb{S}^2)=\left\{f\in L^1(\mathbb{S}^2):f\geq 0 \text{ a.e. }, \int_{\mathbb{S}^2}f\,dp=1\right\}$. All integrals $\int_{\mathbb{S}^2}g(p)\,dp$ use $dp$ to denote the 2-dimensional Hausdorff measure on the sphere. Given the parameter $\eta \in \mathbb{R}$, define the energy functional $\mathcal{F} : \mathcal{P}(\mathbb{S}^2)\times \mathbb{R}\to\mathbb{R}\cup\{+\infty\}$ by 
\begin{equation}\mathcal{F}(f,\eta)=\int_{\mathbb{S}^2}f(p)\ln f(p) -f(p)\ln\left(\int_{\mathbb{S}^2}\left((p\cdot q)^2-\frac{1}{3}\right)f(q)\,dq-\eta \right)\,dp.\end{equation}
$-\ln$ is extended by $+\infty$ for non-positive argument, and the convention $0\times +\infty =0$ is taken. When unambiguous, the dependence of $\mf$ on $\eta$ will be suppressed, so that $\mf(f,\eta)=\mf(f)$. For brevity, given $f\in\mathcal{P}(\mathbb{S}^2)$, define 
\begin{equation}Q=\int_{\mathbb{S}^2}\left(p\otimes p -\frac{1}{3}I\right)f(p)\,dp \in \mathcal{Q}.\end{equation}
\end{definition}

Before any analysis, of course one must ensure that minimisers of the energy actually exist. It will be shown that solutions exist if and only if $\eta<\frac{2}{3}$, although this is due to a lack of finite-energy configurations, rather than minimising sequences being lost due to a lack of lower semicontinuity or coercivity of the functional.

\begin{lemma}\label{lemmaNonEmptyDomain}
There exists some $f \in \mathcal{P}(\mathbb{S}^2)$ with $\mathcal{F}(f)<+\infty$ if and only if $\eta <\frac{2}{3}$. 
\end{lemma}
\begin{proof}
Assume that $\eta\geq\frac{2}{3}$. Then for any $f\in\mathcal{P}(\mathbb{S}^2)$, and corresponding Q-tensor $Q \in \mathcal{Q}$, it must hold that $Qp\cdot p-\eta\leq 0$ for all $p \in\supp(f)\subset\mathbb{S}^2$ due to the eigenvalue constraint on $Q$. In particular, $-\ln(Qp\cdot p-\eta)=+\infty$ for all $p\in\supp(f)$, and $\mathcal{F}(f)=+\infty$. 

Assume $\eta <\frac{2}{3}$. In order to demonstrate that there exists some admissible $f$, in the sense that it has finite energy, it is sufficient, using that $-\ln(\cdot)$ is continuous on its domain, to show that there exists some $f$, so that $\text{supp}(f)\subset \{\eta+\delta \leq Qp\cdot p < \frac{2}{3}\}$, and also having finite entropy. The upper bound is trivial, due to the eigenvalue constraint on the Q-tensor. Let $e \in \mathbb{S}^2$ be arbitrary, and $\epsilon >0$. Define $f^\epsilon = \frac{1}{2\pi \epsilon}\chi_{\{p\cdot e > 1-\epsilon\}}$. A straightforward calculation gives that $f \in \mathcal{P}(\mathbb{S}^2)$, and $Q^\epsilon = \left(1-\frac{3}{2}\epsilon + \frac{1}{2}\epsilon^2\right)\left(e\otimes e -\frac{1}{3}I\right)$. Therefore $Q^\epsilon p \cdot p = \left(1-\frac{3}{2}\epsilon + \frac{1}{2}\epsilon^2\right)\left((p\cdot e)^2-\frac{1}{3}I\right)$. It suffices to show that, for some $\epsilon>0$, $\{p\cdot e > 1-\epsilon\}=\text{supp}(f^\epsilon) \subset \{ Q^\epsilon p \cdot p >\eta+\delta\}= \left\{(p\cdot e)^2 >\frac{2(\eta+\delta)}{2-3\epsilon+\epsilon^2}+\frac{1}{3}\right\}$. It then needs to be shown that there exists some $\epsilon>0$ so that 
\begin{equation}(1-\epsilon)^2- \frac{2(\eta+\delta)}{2-3\epsilon+\epsilon^2}-\frac{1}{3}> 0.\end{equation}
At $\epsilon = 0$, the left-hand side of this inequality is $\frac{2}{3}- (\eta+\delta)$
which by assumption is strictly positive for sufficiently small $\delta$. Then, by continuity at $\epsilon =0$ (the denominator of the rational function is zero only at $\epsilon =1,2$), this implies that for sufficiently small $\epsilon > 0$, the strict inequality holds also. Therefore $\mathcal{F}(f_\epsilon)<+\infty$ for sufficiently small $\epsilon$. 
\end{proof}

\begin{remark}
The previous proposition is a weaker form of \Cref{propAdmissibleQ}, although it is included for its more straightforward proof.
\end{remark}

\begin{lemma}\label{lemmaLowerSemiContinuous}
Let $f_j \in\mathcal{P}(\mathbb{S}^2)$, with $f_j \rightharpoonup f$ in $L^1(\s2)$. Then 
\begin{equation}\liminf\limits_{j \to \infty}\mathcal{F}(f_j)\geq \mathcal{F}(f).\end{equation}
\end{lemma}
\begin{proof}
The Shannon entropy term is trivially lower semicontinuous by convexity. It suffices to show that 
\begin{equation}
\liminf\limits_{j \to \infty}-\int_{\mathbb{S}^2}f_j(p)\ln(Q_jp\cdot p-\eta)\,dp\geq-\int_{\mathbb{S}^2}f(p)\ln(Qp\cdot p-\eta)\,dp.
\end{equation}
For $M \in \mathbb{R}$, let $\ln^M(x)=\max(\ln(x),-M)$. Then note since $Q_j \to Q$, $Q_jp\cdot p-\eta \to Qp\cdot p$ uniformly on $\mathbb{S}^2$, and $-\ln^M(Qp\cdot p-\eta)\to -\ln^M(Qp\cdot p-\eta)$ uniformly. This then implies that 
\begin{equation}
\begin{split}
&\liminf\limits_{j \to \infty}-\int_{\mathbb{S}^2}f_j(p)\ln(Q_jp\cdot p-\eta)\,dp\\
\geq & \liminf\limits_{j \to \infty}-\int_{\mathbb{S}^2}f_j(p)\ln^M(Q_jp\cdot p-\eta)\,dp\\
=&\int_{\mathbb{S}^2}-f(p)\ln^M(Qp\cdot p-\eta)\,dp 
\end{split}
\end{equation}
We then apply the monotone convergence theorem by taking $M\to +\infty$ in the final integral to give 
\begin{equation}
\liminf\limits_{j \to \infty}-\int_{\mathbb{S}^2}f_j(p)\ln(Q_jp\cdot p-\eta)\,dp\geq \int_{\mathbb{S}^2}-f(p)\ln(Qp\cdot p-\eta)\,dp .
\end{equation}
\end{proof}

\begin{proposition}\label{theoremExistence}
There exists a minimiser of $\mathcal{F}$ if and only if $\eta<\frac{2}{3}$. 
\end{proposition}
\begin{proof}
The proof will follow a standard direct method argument (e.g. \cite{dacorogna2007direct}). The eigenvalue constraint on $Q$ gives that $Qp\cdot p-\eta$ is bounded from above, so that $-\ln (Qp\cdot p-\eta)$ is bounded from below. Similarly, the Shannon entropy is bounded from below, so we have a minimising sequence $(f_j)_{j\in\mathbb{N}}$ provided $\text{dom}(F)\neq \emptyset$, which is precisely when $\eta <\frac{2}{3}$ from  \Cref{lemmaNonEmptyDomain}. Since $\mathcal{F}(f_j)$ is bounded, this implies that $\int_{\mathbb{S}^2}f_j(p)\ln f_j(p)\,dp$ is bounded, so there exists a subsequence (not relabelled) which converges weakly in $L^1$ to some $f^* \in \ps2$. Finally, $\mathcal{F}$ is $L^1(\s2)$ weakly lower semicontinuous by  \Cref{lemmaLowerSemiContinuous}, completing the proof.
\end{proof}

One might hope to find minimisers by taking smooth variations, although the first immediate issue is that if $f \in \mathcal{P}(\mathbb{S}^2)$ is not bounded away from zero, it cannot be guaranteed that an arbitrary variation of the form $f+\epsilon\varphi$ is in the domain of $\mathcal{F}$. Let $\mathcal{P}_+(\mathbb{S}^2)=\left\{ f \in \mathcal{P}(\mathbb{S}^2): \exists M>0, \, \frac{1}{M}\leq f(p) \leq M \text{ a.e.}\right\}$ denote the set of probability distributions bounded away from zero and infinity. Note that if $f \in \mathcal{P}_+(\mathbb{S}^2) \cap \text{dom}(\mf)$, then $v_{\min}(Q)>\eta$, and in particular if $\text{dom}(\mathcal{F})\cap\mathcal{P}_+(\mathbb{S}^2)\neq \emptyset$, then $\eta<0$. The converse also holds, since the isotropic state $f_U(p)=\frac{1}{4\pi}$ satisfies $\mathcal{F}(f_U)<+\infty$ if and only if $\eta<0$. This set is significant because it contains all of the probability distributions where arbitrary variations of the form $f+\epsilon\varphi$ can be taken, with $\varphi \in L^\infty(\mathbb{S}^2)$, $\int_{\mathbb{S}^2}\varphi(p)\,dp=0$. In particular, with respect to the $L^\infty$ topology, $\mathcal{P}_+(\mathbb{S}^2)\cap \text{dom}(\mf)$ is open in $\mathcal{P}(\mathbb{S}^2)$.

\begin{proposition}\label{propIsotropicGeneral}
Let $\eta <0$. The first and second variations of $\mathcal{F}$ about $f \in \mathcal{P}_+ \cap \text{dom}(J)$ are given by 
\begin{equation}
\begin{split}
\delta\mathcal{F}(f)[\phi]=&\int_{\mathbb{S}^2}\phi\ln(f)+\phi -\phi\ln\left(Qp\cdot p-\eta\right)-\frac{fAp\cdot p}{Qp\cdot p-\eta}\,dp,\\
\delta^2\mathcal{F}(f)[\phi,\phi]=&\int_{\mathbb{S}^2}\frac{1}{f}\left(\phi-\frac{fAp\cdot p}{Qp\cdot p-\eta}\right)^2\,dp,
\end{split}
\end{equation}
where $\phi \in L^\infty(\s2)$, $A=\int_{\mathbb{S}^2}\left(p\otimes p-\frac{1}{3}I\right)\phi(p)\,dp\in\sm3$ and $\int_{\s2}\phi(p)\,dp=0$.

In particular, $\mathcal{F}$ is convex when restricted to $\mathcal{P}_+(\mathbb{S}^2)\cap \text{dom}(J)$. Furthermore, the first variation of $\mathcal{F}$ vanishes at $f_U=\frac{1}{4\pi}$, and the second variation at $f_U$ is strictly positive if $\eta \neq -\frac{2}{15}$
\end{proposition}
\begin{proof}
Let $\phi \in L^\infty(\s2)$, with $\int_{\mathbb{S}^2}\phi=0$. Denote $A=\int_{\mathbb{S}^2}\phi(p)\left(p\otimes p-\frac{1}{3}I\right)\,dp$. For readability the $p$ dependence of $f,\phi$ will be implicit. The variations are readily calculated as
\begin{equation}
\begin{split}
\mathcal{F}(f+\epsilon \phi) = & \int_{\mathbb{S}^2} (f+\epsilon \phi)\ln\left(\frac{(f+\epsilon \phi)}{(Q+\epsilon A)p\cdot p-\eta}\right)\,dp,\\
\frac{d}{d\epsilon}\mathcal{F}(f+\epsilon \phi)=& \int_{\mathbb{S}^2}\phi\ln(f+\epsilon \phi)+\phi -\phi\ln\left((Q+\epsilon A)p\cdot p-\eta\right)\\
&-\frac{(f+\epsilon\phi)Ap\cdot p}{(Q+\epsilon A)p\cdot p-\eta}\,dp,\\
\frac{d^2}{d\epsilon^2}\mathcal{F}(f+\epsilon\phi)=& \int_{\mathbb{S}^2}\frac{\phi^2}{f+\epsilon\phi}-2\frac{\phi Ap\cdot p}{(Q+\epsilon A)\cdot p-\eta} \\
&+\frac{(f+\epsilon\phi)(Ap\cdot p)^2}{((Q+\epsilon A)p\cdot p-\eta)^2}\,dp,\\
\left.\frac{d}{d\epsilon}\mathcal{F}(f+\epsilon\phi)\right|_{\epsilon=0}=& \int_{\mathbb{S}^2}\phi\ln(f)+\phi -\phi\ln\left(Qp\cdot p-\eta\right)\\
&-\frac{fAp\cdot p}{Qp\cdot p-\eta}\,dp,\\
\left.\frac{d^2}{d\epsilon^2}\mathcal{F}(f+\epsilon \phi)\right|_{\epsilon=0}=& \int_{\mathbb{S}^2} \frac{\phi^2}{f}-\frac{2\phi Ap\cdot p}{Qp\cdot p-\eta}+\frac{f(Ap\cdot p)^2}{(Qp\cdot p-\eta)^2}\,dp\\
=&\int_{\mathbb{S}^2}\frac{1}{f}\left(\phi-\frac{fAp\cdot p}{Qp\cdot p-\eta}\right)^2\,dp.
\end{split}
\end{equation}
The convexity of $\mathcal{F}$ on the restricted set then follows since $\mathcal{P}_+(\mathbb{S}^2)$ is convex and open in $\text{dom}(\mf)$ with respect to the strong $L^\infty(\mathbb{S}^2)$ topology, with positive second variation on its domain. Taking $f(p)=\frac{1}{4\pi}$ gives $Q=0$, and the first variation is 
\begin{equation}\int_{\mathbb{S}^2}\phi(p)\left(1-\ln(4\pi)-\ln(-\eta)\right)\,dp+\frac{1}{3\eta}I\cdot A=0.\end{equation}
The second variation at $f=f_U$ is then given as 
\begin{equation}4\pi\int_{\mathbb{S}^2}\left(\phi(p)+\frac{1}{4\pi\eta}Ap\cdot p\right)^2\,dp.\end{equation}
This is strictly positive unless 
\begin{equation}(-4\pi\eta)\phi(p)=Ap\cdot p=\int_{\mathbb{S}^2}\left((p\cdot q)^2-\frac{1}{3}\right)\phi(q)\,dq\end{equation}
almost everywhere. This is eigenvalue problem is implicitly solved in \cite{fatkullin2005critical}, since they establish that the linear operator on the right is, when defined for $L^2$ complex valued functions, $-\frac{8\pi}{15}$ multiplied by the projection onto the set of spherical harmonics of order two. Hence the eigenvalue problem has only the trivial solution $-4\pi\eta=0$, at which point $f_U$ is inadmissible, and $-4\pi\eta=\frac{8\pi}{15}\Rightarrow \eta=-\frac{2}{15}$. Therefore the second variation of $\mathcal{F}$ at $f_U=\frac{1}{4\pi}$ is strictly positive unless $\eta=-\frac{2}{15}$
\end{proof}

While for $\eta<0$, $\mathcal{F}$ is convex on a subset of its domain, we now show that $\mathcal{F}$ is generally not convex. This global result will later be strengthened to a result demonstrating a lack of local convexity at non-trivial local minimisers in \Cref{propNotLocallyConvex}.

\begin{proposition}\label{propositionNonConvexDomain}
Let $0\leq \eta <\frac{2}{3}$. Then the effective domain of $\mathcal{F}$ is not convex and in particular $\mathcal{F}$ itself is not convex.
\end{proposition}
\begin{proof}
Let $f_1 \in \mathcal{P}(\mathbb{S}^2)\cap\text{dom}(\mathcal{F})\neq \emptyset$. Without loss of generality take its Q-tensor, $Q_1$, to be diagonal with eigenvalues $q_1,q_2,q_3$ corresponding to basis vectors $e_1,e_2,e_3$ respectively. Let $f_2,f_3$ be defined by rotations acting on $f_1$, so that their corresponding Q-tensors $Q_2$ and $Q_3$ have the same eigenbasis but with permuted eigenvectors. Explicitly, $Q_2e_i=q_{i+1}e_i$ and $Q_3e_i=q_{i+2}e_i$, with indices taken modulo 3. In this case, $\frac{1}{3}(Q_1+Q_2+Q_3)e_i = \frac{1}{3}(q_1+q_2+q_3)=0$. If the domain of $\mathcal{F}$ were convex, then $\frac{1}{3}(f_1+f_2+f_3)$ must have finite energy. However, since the Q-tensor of $\frac{1}{3}(f_1+f_2+f_3)$ is zero and $\eta\geq 0$, this implies that $\mathcal{F}\left(\frac{1}{3}(f_1+f_2+f_3)\right)=+\infty$, giving a contradiction.
\end{proof}

\begin{proposition}\label{propositionIsotropicStateGlobalMin}
For all $\eta<0$, the isotropic state is a global minimiser on $\mathcal{P}_+(\mathbb{S}^2)\cap \text{dom}(\mathcal{F})$, and the unique global minimiser on this set if $\eta\neq -\frac{2}{15}$.
\end{proposition}
\begin{proof}
Since $\mathcal{F}$ is convex on the restricted set, which is open in $L^\infty$, and the first variation vanishes, $f_U$ must be a global minimiser. Furthermore, since the second variation is strictly positive for $\eta\neq -\frac{2}{15}$, this implies it is a strict local minimiser, and therefore a unique global minimiser on the restricted set.
\end{proof}

\begin{corollary}\label{corollaryOnlyIsotropicVariation}
Unless $\eta=-\frac{2}{15}$, the only $L^\infty$ local minimiser that can be found by solving 
\begin{equation}\left.\frac{d}{dt}\mathcal{F}(f+t\phi)\right|_{t=0}=0\end{equation}
for all $\phi \in L^\infty(\mathbb{S}^2)$ with $\int_{\mathbb{S}^2}\phi(p)\,dp=0$ is the isotropic state, and only when $\eta<0$. 
\end{corollary}

\begin{remark}
In \Cref{corollaryOnlyIsotropicVariation} we have seen that support conditions raise difficulties in finding solutions by taking variations and solving
\begin{equation}
\left.\frac{d}{d\epsilon}\mathcal{F}(f^*+\epsilon \phi)\right|_{\epsilon=0}=0.
\end{equation}
Furthermore, \Cref{propositionNonConvexDomain} also states that attempting variational inequalities and finding solutions by considering 
\begin{equation}
\left.\frac{d}{d\epsilon}\mathcal{F}((1-\epsilon)f^*+\epsilon f)\right|_{\epsilon=0}\geq 0
\end{equation}
is not generally possible either. Together, these results imply that non-standard techniques will be needed for finding local minimisers.
\end{remark}

\begin{proposition}\label{propEta215Drama}
Let $\eta=-\frac{2}{15}$, and $V$ denote subspace of real valued functions in the span of the second order spherical harmonics. Then if $f=\frac{1+u}{4\pi}$ for $u \in V$, $\inf\limits_{p\in\mathbb{S}^2}u(p)>-1$, we have $\mathcal{F}(f)=\mathcal{F}(f_u)=\ln\left(\frac{15}{8\pi}\right)$. Furthermore, all such $f$ have vanishing first variation. 
\end{proposition}
\begin{proof}
Let $Q$ denote the corresponding $Q$ tensor for $f$. Then 
\begin{equation}
\begin{split}
Qp\cdot p=&\int_{\mathbb{S}^2}\left((p\cdot q)^2-\frac{1}{3}\right)\frac{1+u(q)}{4\pi}\,dq\\
=&\frac{1}{4\pi}\int_{\mathbb{S}^2}\left((p\cdot q)^2-\frac{1}{3}\right)u(q)\,dq\\
=&\frac{2}{15}u(p).
\end{split}
\end{equation}
Therefore substituting this into the energy, 
\begin{equation}
\begin{split}
\mathcal{F}(f)=& \int_{\mathbb{S}^2}f(p)\ln f(p)-f(p)\ln \left(\frac{2}{15}u(p)+\frac{2}{15}\right)\,dp\\
=& \int_{\mathbb{S}^2}f(p)\ln f(p) - f(p) \ln \left(\frac{8\pi}{15}\frac{1+u(p)}{4\pi}\right)\,dp\\
=& \int_{\mathbb{S}^2}f(p) \ln f(p) - f(p) \ln \left(\frac{8\pi}{15}f(p)\right)\,dp\\
=& \int_{\mathbb{S}^2}f(p)\ln\left( \frac{15}{8\pi}\right)\,dp=\ln\left(\frac{15}{8\pi}\right).
\end{split}
\end{equation}
Finally, using that $Qp\cdot p-\eta=\frac{8\pi}{15}f(p)$, we substitute this into the equation for the first variation in \eqref{eqFirstSecondVariation}, to give 
\begin{equation}
\begin{split}
\delta\mathcal{F}(f)[\phi]=&\int_{\mathbb{S}^2}\phi \ln (f)+\phi-\phi\ln(Qp\cdot p-\eta)-\frac{fAp\cdot p}{Qp\cdot p-\eta}\,dp\\
=&\int_{\mathbb{S}^2}\phi\left(1+ \ln \left(\frac{15}{8\pi}\right)\right)-\frac{15Ap\cdot p}{8\pi}\,dp\\
=&0,
\end{split}
\end{equation}
since $\phi$ integrates to $1$ and $\int_{\mathbb{S}^2}Ap\cdot p\,dp=4\pi \text{Trace}(A)=0$. 
\end{proof}

\begin{remark}
The numerical results of \cite{zheng2016density} suggest that these are not global minimisers of the energy at $\eta=-\frac{2}{15}$.
\end{remark}

\begin{corollary}\label{corollaryIsotropicDrama}
When $\eta=-\frac{2}{15}$, the isotropic state is no longer a strict  $L^\infty$-local minimiser, but it is an $L^\infty$- local minimiser. 
\end{corollary}
\begin{proof}
From the previous results we have that the second variation is only degenerate for variations in $V$, which have the same energy as the isotropic state.
\end{proof}

Loosely speaking, the results in this section imply that taking variations to obtain the Euler-Lagrange equation can only find trivial solutions. In particular, when $\eta \in \left[0,\frac{2}{3}\right)$, it cannot provide any results even though minimisers exist. Rather than tackle the full minimisation problem, following the spirit of \cite[Subsection 4.1]{taylor2015maximum}, the minimisation problem will instead be split into two manageable steps.

\section{The auxiliary problem and the Euler-Lagrange equation for global minimisers}
\label{secAuxiliaryProblem}
Given $Q \in \mathcal{Q}$, define 
\begin{equation}A(Q)=\left\{f \in \ps2 : \int_{\mathbb{S}^2}\left(p\otimes p-\frac{1}{3}I\right)f(p)\,dp=Q\right\}\end{equation}
to be the admissible set for $Q$. Then the minimisation problem can be split into 
\begin{equation}
\min\limits_{f \in\ps2}\mathcal{F}(f,\eta)=\min\limits_{Q \in \mathcal{Q}}\left(\min\limits_{f \in A(Q)}\mathcal{F}(f,\eta)\right).
\end{equation}
The interior minimisation problem over $\mathcal{A}(Q)$ will be referred to as the auxiliary problem. Since $\mathcal{F}$ is strictly convex on $\mathcal{A}(Q)$, this problem is much more readily tackled, drawing mainly on results from Borwein and Lewis \cite{borwein1991duality} and Taylor \cite{taylor2015maximum}. Once this simpler problem has been analysed, it remains to consider the finite-dimensional problem of minimising the macroscopic auxiliary function over admissible Q-tensors.

\begin{definition}
Define the auxiliary function $J:\mathcal{Q}\times\mathbb{R}\to\mathbb{R}\cup\{+\infty\}$ by 
\begin{equation}J(Q,\eta)=\inf\limits_{f \in A(Q)}\mathcal{F}(f,\eta).\end{equation}
For fixed $\eta$, define the set 
\begin{equation}E_Q=\{p \in\mathbb{S}^2:Qp\cdot p>\eta\}.\end{equation}
\end{definition}
Note that if $f \in \ps2$ with Q-tensor $Q$ and $\mathcal{F}(f)<+\infty$, then $\{p \in \mathbb{S}^2:f(p)>0\}\subset E_Q$ up to a set of measure zero. As before, when the dependence of $J$ on $\eta$ is unambiguous, the dependence will be suppressed so that $J(Q,\eta)=J(Q)$.

\begin{proposition}\label{propFormOfJ}
Let $Q \in \text{dom}(J)$. Then there exists a unique solution to $\min\limits_{f \in A(Q)}\mathcal{F}(f)$, given by 
\begin{equation}f_Q(p)=\frac{1}{Z}\exp(\Lambda(Q) p\cdot p)\max(Qp\cdot p-\eta,0)\end{equation}
for all $p \in \mathbb{S}^2$, where $Z>0$ is a normalising constant depending on $(Q,\eta)$, and $\Lambda(Q) \in \sm3$ maximises the dual objective function $F:\text{dom}(J)\times \sm3 \to \mathbb{R}$ given by
\begin{equation}F(Q,\lambda)=\lambda \cdot Q - \ln\left(\int_{\mathbb{S}^2}\exp(\lambda p\cdot p)\max(Qp\cdot p-\eta,0)\,dp\right).\end{equation}
In particular $J(Q)=F(Q,\Lambda(Q))=\max\limits_{\lambda \in\sm3} F(Q,\lambda).$
\end{proposition}
\begin{proof}
Existence follows by the same argument as \Cref{theoremExistence}, noting that $A(Q)$ is weakly closed, under the assumption that $Q\in \text{dom}(J)$, which ensures the admissible set is non-empty. Uniqueness follows from the strict convexity of $\mathcal{F}$ when restricted to $A(Q)$. Recall that $\supp(f)\subset E_Q =\{p\in\mathbb{S}^2:Qp\cdot p>\eta\}$, else the energy is infinite. The minimisation problem can then be written as
\begin{equation}
\begin{split}
\text{minimise } &\int_{E_Q}f(p)\ln f(p)-f(p)\ln (Qp\cdot p-\eta)\,dp,\\
\text{subject to }0\leq & f(p) \text{ a.e.},\\
1=&\int_{E_Q}f(p)\,dp,\\
Q=&\int_{E_Q}f(p)\left(p\otimes p-\frac{1}{3}I\right)\,dp.
\end{split}
\end{equation}
Define $\tilde{f}(p)=\frac{1}{Qp\cdot p-\eta} f(p)$ on $E_Q$. Define the measure $\mu_Q$ on $E_Q$ as $d\mu_Q(p)=(Qp\cdot p-\eta)dp$. Then the minimisation problem is equivalent to 
\begin{equation}\label{eqEasyMin}
\begin{split}
\text{minimise } & \int_{E_Q}\tilde{f}(p)\ln \tilde{f}(p)\,d\mu_Q(p),\\
\text{subject to }0\leq & \tilde{f}(p) \text{ a.e.},\\
1 =& \int_{E_Q} \tilde{f}(p)\,d\mu_Q(p),\\
Q=& \int_{E_Q} \tilde{f}(p)\left(p\otimes p-\frac{1}{3}I\right)\,d\mu_Q(p).
\end{split}
\end{equation}
This is a straightforward entropy minimisation subject to linear constraints, and \cite{borwein1991duality} can be applied. The only technicality that needs to be addressed for the results of Borwein and Lewis to be applied is that the so-called {\it pseudo-Haar} condition is satisfied by the constraint functions. By \cite{taylor2015maximum}, since the constraint functions are analytic on the sphere, and the non-null subsets of $E_Q$ with respect to $\mu_Q$ are also non-null subsets with respect to $\mathcal{H}^2$, this is not problematic. The solution is then given by 
\begin{equation}\tilde{f}(p)=\exp(\Lambda(Q) p\cdot p+\alpha^*-1)\end{equation}
on $E_Q$ for $\Lambda(Q) \in \sm3$ and $\alpha^* \in \mathbb{R}$ that maximise the dual objective function
\begin{equation}(\lambda,\alpha)\mapsto\lambda \cdot Q + \alpha - \ln \left(\int_{E_Q}\exp(\alpha -1 +\lambda p \cdot p) \mu_Q(p)\,dp\right).\end{equation}
Since the objective function is smooth and concave in $(\lambda,\alpha)$, we can eliminate $\alpha$ by setting the derivative with respect to $\alpha$ to zero, which gives 
\begin{equation}\lambda \cdot Q-\ln \left(\int_{E_Q}\exp(\lambda p\cdot p)\,d\mu_Q(p)\,dp\right),\end{equation}
which the form given in the statement. Finally, $f$ can be reclaimed from $\tilde{f}$ as 
\begin{equation}f(p)=(Qp\cdot p-\eta)\tilde{f}(p)=\frac{1}{Z}\exp(\Lambda(Q) p\cdot p)(Qp\cdot p-\eta),\end{equation}
on $E_Q$ with $Z=\exp(1-\alpha^*)$. Noting that $f$ must be zero when $Qp\cdot p\leq \eta$ provides the form in the statement. 
\end{proof}

In general the domain of $J$ will depend on $\eta$. For example, it is immediate that $J(0)=+\infty$ if $\eta \geq 0$, but $J(0)$ is finite otherwise by taking the uniform distribution $p\mapsto \frac{1}{4\pi}$. The domain can fortunately be explicitly determined. The key step is to establish that $J(Q)$ is finite if and only if $Q$ lives in a particular convex set, which admits an explicit representation in terms of supporting hyperplanes. First we include a lemma necessary for the proof.

\begin{lemma}\label{lemmaRotationDerivative}
Let $A \in \sm3$. If $p$ is a local maximum of $Ap\cdot p$, then $p$ is an eigenvector of $A$. 
\end{lemma}
\begin{proof}
Consider the function $x\mapsto \frac{1}{|x|^2}Ax\cdot x$ for $x \in \mathbb{R}^3\setminus\{0\}$. If $x$ is a local maximum of this function, then $p=\frac{1}{|x|}$ is a local maximum of $\tilde{p}\mapsto Ap\cdot p$. In particular, it suffices to show that if $\nabla \frac{1}{|x|^2}Ax\cdot x =0$, then $Ax\parallel x$. The derivative is readily computed as $ \nabla \frac{1}{|x|^2}Ax\cdot x = -\frac{2}{|x|^4}(Ax\cdot x)x +\frac{2}{|x|^2}Ax$, so if the derivative vanishes $Ax=\frac{Ax\cdot x}{|x|^2}x$ and the result follows. 
\end{proof}

\begin{proposition}
\label{propAdmissibleQ}
The domain of $J$ is given by $\mathcal{Q}\cap\{Q \in \sm3 : |Q|^2>\eta\}$. In particular, $\text{dom}(J)$ is open.
\end{proposition}
\begin{proof}
Given $Q \in \mathcal{Q}$, taken without loss of generality to be in its diagonal frame, define 
\begin{equation}\mathcal{Q}_Q=\left\{ \int_{E_Q}\left(p\otimes p-\frac{1}{3}I\right)f(p)\,d\mu_Q(p) : f \in \mathcal{P}(E_Q;\mu_Q)\right\},\end{equation}
where $\mathcal{P}(E_Q;\mu_Q)$ explicitly denotes that integration is with respect to $\mu_Q$ as given in \Cref{propFormOfJ}. Using the results of Borwein and Lewis \cite{borwein1991duality} and the equivalent minimisation problem given in \Cref{eqEasyMin}, $J(Q)<+\infty$ if and only if $Q \in \mathcal{Q}_Q$.  By \cite{taylor2015maximum}, $\mathcal{Q}_Q$ is an open, convex set and we have $Q_0 \in \sm3$ with $Q_0 \in \mathcal{Q}_Q$ if and only if, for all $A \in \sm3\cap\setminus\{0\}$, 
\begin{equation}A\cdot Q_0 < \sup\limits_{p \in E_Q} Ap\cdot p.\end{equation}

If $Q \in \text{dom}(J)$ then $Q \in \mathcal{Q}$, so the eigenvalue constraint must be satisfied. By taking $A=-Q$ this implies that 
\begin{equation}(-Q)\cdot Q< \sup\limits_{p\in E_Q} (-Q)p\cdot p=-\eta+\sup\limits_{p\in E_Q} (\eta-Qp\cdot p)\,dp \leq -\eta,\end{equation}
so multiplying both sides by $-1$ gives that $|Q|^2>\eta$.

Now assume that $|Q|^2>\eta$ and $v_{\min}(Q)>-\frac{1}{3}$, and take $A \in \sm3\cap \setminus\{0\}$. Then the aim is to show that there exists $p \in E_Q$ so that
$Q\cdot A <Ap\cdot p$. Rather than finding some $p$ satisfying the strict inequality, some $p \in E_Q$ will be found so that equality holds, and then a perturbation argument will be used to show that such a $p$ is not a maximiser of $Ap\cdot p$. 

Take $p^*$ to be given componentwise by $p_i^*=\pm\sqrt{v_i(Q)+\frac{1}{3}}$, where $v_i(Q)$ is the eigenvalue of $Q$ corresponding to $e_i$. Note that the eigenvalue constraint on $Q$ gives that this is well defined, and the tracelessness condition gives that $|p^*|^2=1$. Each $p_i$ admits a choice of sign, and this is unimportant with the exception that they must be chosen so that $p^*$ is not an eigenvector of $A$. Such a choice will always exist however. To see this, let $\tilde{p}^*=(I-2e_i\otimes e_i)p^*$, that is the $i$-th coordinate changes sign. Then $\tilde{p}^*\cdot p^*=1-2(p_i^*)^2$. If $\tilde{p}^*$ and $p^*$ are eigenvectors of $A$, then this dot product must either be $1$, $-1$ or $0$ by orthogonality. In no case can this be equal to $1$ or $-1$, since $|p_i^*|\neq 0,1$ due to the eigenvalue constraint on $Q$. There must be at least one choice of $i$ where this is non-zero, since $1-2(p_i^*)^2 =\frac{1}{3}-2v_i(Q)$. Therefore if all are zero, then Q has three equal but non-zero eigenvalues, which is impossible.

It remains to be verified that $p^* \in E_Q$, which holds since 
\begin{equation}Qp^*\cdot p^* =\sum\limits_{i=1}^3v_i(Q)\sqrt{v_i(Q)+\frac{1}{3}}^2=\sum\limits_{i=1}^3 v_i(Q)^2+\frac{v_i(Q)}{3}=|Q|^2 > \eta.\end{equation}
Finally, the desired equality holds since
\begin{equation}Q\cdot A = \sum\limits_{l=1}^3 Q_{ll}A_{ll}= \sum\limits_{l=1}^3 (p_l^*)^2 A_{ll}=Ap^*\cdot p^*.\end{equation}

If it can then be shown that $p^*$ is not a maximum of $p\mapsto Ap\cdot p$ on $E_Q$, then the result is proven. Since $p^*$ is in $E_Q$, which is open in $\mathbb{S}^2$, a perturbation argument can be used and we can simply demonstrate that $p^*$ is not a local maximiser on the sphere. From \Cref{lemmaRotationDerivative}, we know that any maximiser of $p\mapsto Ap\cdot p$ must be an eigenvector of $A$, however by our construction this is not the case. Therefore there exists some $p \in \mathbb{S}^2$ with $Ap\cdot p > Ap^*\cdot p^* = Q\cdot A$, and the result follows. 
\end{proof}

\begin{proposition}
$J$ is convex on the set $\{Q \in \text{dom}(J): v_{\min}(Q)>\eta\}$. 
\end{proposition}
\begin{proof}
If $v_{\min}(Q)>\eta$, then $Qp\cdot p-\eta>0$ for all $p \in \mathbb{S}^2$, and $E_Q=\mathbb{S}^2$. For such $Q$,
\begin{equation}
\begin{split}
J(Q)=& \max\limits_{\lambda \in \sm3}\lambda \cdot Q - \ln \left(\int_{\mathbb{S}^2}\exp(\lambda p\cdot p)(Qp\cdot p-\eta)\,dp\right)\\
=& \max\limits_{\lambda \in\sm3}\lambda \cdot Q -\ln\left(Q\cdot \int_{\mathbb{S}^2}\exp(\lambda p\cdot p)p\otimes p\,dp -\eta\int_{\mathbb{S}^2}\exp(\lambda p\cdot p)\,dp\right).
\end{split}
\end{equation}
By writing it this way, it is clear that $J$ can be written as the maximum of a set of convex functions, which follows immediately from the convexity of the negative logarithm. Therefore $J$ is convex. 
\end{proof}

\begin{proposition}[Uniform blow up of $J$]
\label{propJBlowUp}
For $Q \in \text{dom}(J)$, $J(Q)\geq \psi_s(Q)-\ln (|Q|^2-\eta)$, where $\psi_s$ is the Ball-Majumdar singular potential \cite{ball2010nematic} given by 
\begin{equation}\psi_s(Q)=\min\limits_{f \in \mathcal{A}(Q)}\int_{\s2}f(p)\ln f(p)\,dp.\end{equation} In particular, $J$ blows up to $+\infty$ uniformly at the boundary of its domain.
\end{proposition}
\begin{proof}
Using Jensen's inequality, 
\begin{equation}
\begin{split}
\min\limits_{f \in A(Q)}\mathcal{F}(f)=&\min\limits_{f \in A(Q)}\int_{\mathbb{S}^2}f(p)\ln f(p)-f(p)\ln (Qp\cdot p-\eta)\,dp\\
\geq &\min\limits_{f \in A(Q)}\int_{\mathbb{S}^2}f(p)\ln f(p) - f(p) \,dp - \ln\left(\int_{\mathbb{S}^2}\left(Qp\cdot p-\eta\right)f(p)\,dp\right)\\
=& \min\limits_{f \in A(Q)}\int_{\mathbb{S}^2}f(p)\ln f(p) \,dp - \ln (|Q|^2-\eta)\\
=& \psi_s(Q)-\ln (|Q|^2-\eta).
\end{split}
\end{equation}
If $Q_j \to \partial\text{dom}(J)$, then this means either $v_{\min}(Q)\to-\frac{1}{3}$ , in which case $\psi_s(Q_j)\to+\infty$ \cite{ball2010nematic}, or $|Q_j|^2 \to \eta$, in which case the logarithmic term blows up. 
\end{proof}

This blow up of $J$ at the boundary of its domain serves to ensure that minimising sequences cannot be lost at the boundary. More precisely, if $J(Q_j)\to \min\limits_{Q \in \mathcal{Q}}J(Q)$, then we must have some $\delta>0$ so that $\text{dist}(Q_j,\partial\text{dom}(J))>\delta$ for all $j$. Furthermore, combining this with the continuity of $J$ which will be given in \Cref{propJDerivatives}, this gives that the sublevel sets $J^{-1}\left((-\infty,M]\right)$ for $M\in\mathbb{R}$ are compact. 

\begin{proposition}\label{propQLambdaEigenvectors}
$\Lambda$ is a frame-indifferent function of $Q$, and if $e\in\mathbb{S}^2$ is an eigenvector of $Q \in \text{dom}(J)$, then $e$ is an eigenvector of the corresponding maximiser of the dual problem $\Lambda(Q)$, and the converse holds if $\Lambda(Q) \neq 0$. 
\end{proposition}
\begin{proof}
By writing the dual objective function as 
\begin{equation}F(Q,\lambda)=\lambda \cdot Q -\ln\left(\int_{\mathbb{S}^2}\exp(\lambda p\cdot p)\max(Qp\cdot p-\eta,0)\,dp\right),\end{equation}
it is immediate that $F$ is frame indifferent, so that for all $R \in \text{SO(3)}$, $F(RQR^T,R\lambda R^T)=F(Q,\lambda)$. In particular, combined with the uniqueness of maximisers, implies that $\Lambda(RQR^T)=R\Lambda(Q)R^T$. Since $\Lambda$ is frame indifferent, it can be written as $\Lambda(Q)=\sum\limits_{i=0}^3 g_i(Q)Q^i$, with $g_i$ scalar-valued frame-indifferent functions of $Q$, so that if $e$ is an eigenvector of $Q$, $e$ must be an eigenvector of $\Lambda(Q)$. since $\text{Tr}(\Lambda(Q))=0$, this implies that the decomposition can be written as $\Lambda(Q)=g_2(Q)\left(Q^2-\frac{|Q|^2}{3}I\right)+g_0(Q)Q$. Both terms in the sum have exactly the same eigenbasis as $Q$ unless they are zero. 
\end{proof}

\begin{remark}
The previous result leaves the possibility that $\Lambda(Q)$ may have eigenvectors that $Q$ does not if $\Lambda(Q) =0$. By considering the example when $\eta=-\frac{2}{15}$ in \Cref{propEta215Drama}, there are constructed examples such that $Q\neq 0$, but $\Lambda(Q)=0$ which demonstrate that $\Lambda(Q)$ can have eigenvectors that $Q$ does not. However, if $\Lambda(Q)$ is uniaxial and non-zero, then $Q$ must be uniaxial and non-zero also.
\end{remark}

\begin{corollary}

If $f_Q$ is the optimal energy distribution corresponding to $Q$, then for all $R \in \text{SO}(3)$, $f_{RQR^T}(p)=f_Q(Rp)$. In particular, if $RQR^T=Q$, then $f_Q(Rp)=f_Q(p)$. Furthermore $J$ is frame indifferent so that $J(RQR^T)=J(Q)$.
\end{corollary}
\begin{proof}
This follows immediately since $f_Q$ can be written as 
\begin{equation}f_Q(p)=\frac{1}{Z}\exp(\Lambda(Q)p\cdot p)\max(Qp\cdot p-\eta,0).\end{equation}
Frame indifference of $J$ then follows since if $[Rf](p)=f(Rp)$, then $J(RQR^T)=\mathcal{F}(f_{RQR^T})=\mathcal{F}([Rf_Q])=\mathcal{F}(f_Q)=J(Q)$, using that $\mathcal{F}$ is frame indifferent.
\end{proof}

At face-value, the dual maximisation problem $\max\limits_{\lambda \in \sm3} F(Q,\lambda)$ is over a five-dimensional vector space. However by fixing $Q$ in its diagonal frame and using the previous results, it is therefore possible to only consider $\lambda$ in the same diagonal frame. In particular, the maximisation is only over a two-dimensional vector space, which is advantageous if one wishes to calculate $J$ numerically via the dual optimisation problem.

In order to establish smoothness properties of $J$, the first step will be to establish smoothness of the map $\Lambda: \text{dom}(J)\to\sm3$. This will be done using an implicit function theorem argument on the relation
\begin{equation}0=G(Q,\Lambda(Q))=Q-\frac{\int_{E_Q}\exp(\Lambda(Q) p\cdot p)(Qp\cdot p-\eta)\left(p\otimes p-\frac{1}{3}I\right)\,dp}{\int_{E_Q}\exp(\Lambda(Q) p\cdot p)(Qp\cdot p-\eta)\,dp}.\end{equation} 
Before the implicit function theorem can be used, it must first be established that $G$ is $C^1$ on $\text{dom}(J)\times\sm3$. We can illustrate the techniques required in a simpler setting in order to establish stronger regularity results on the set $\{Q \in \mathcal{Q}:v_{\min}(Q)>\eta\}$. For the following, we take $\Lambda, J$ to depend explicitly on $\eta$ also. 

\begin{proposition}
$\Lambda$ and $J$ are $C^\infty$ on $\{(Q,\eta)\in\mathcal{Q}\times\mathbb{R}: |Q|^2>\eta\, ,\,Q \in \mathcal{Q}: v_{\min}(Q)>\eta\}$.
\end{proposition}
\begin{proof}
First note that that on $\text{dom}(J)\cap \{Q \in \mathcal{Q}: v_{\min}(Q)>\eta\}$, $E_Q=\mathbb{S}^2$, and $Q,\Lambda(Q)$ are related by
\begin{equation}0=Q-\frac{1}{\int_{\mathbb{S}^2}\exp(\Lambda p\cdot p)(Qp\cdot p-\eta)\,dp}\int_{\mathbb{S}^2}\exp(\Lambda p\cdot p)(Qp\cdot p-\eta)\left(p\otimes p-\frac{1}{3}I\right)\,dp.\end{equation}
Since there is no issue with the domain of integration, and the integrand is $C^\infty$, this then gives that $G$ is $C^\infty$ for $Q,\Lambda(Q)$ in the given subdomain. The derivative of $G$ with respect to $\lambda$ at $(Q,\Lambda(Q))$ is given by 
\begin{equation}\label{eqVariance}
\int_{\s2}\left(p\otimes p-\frac{1}{3}I\right)\otimes \left(p\otimes p-\frac{1}{3}I\right)f_Q(p)\,dp - Q\otimes Q,
\end{equation}
which is strictly positive definite by Cauchy-Schwarz and the fact that the linearly independent components of $p\otimes p-\frac{1}{3}$ and the constant function form a pseudo-Haar set. Therefore the implicit function theorem gives that $\Lambda$ is a $C^\infty$ function of $Q$. Using that$J(Q)=F(Q,\Lambda(Q))$, where $F$ is also $C^\infty$ on the given subdomain implies that $J$ is a $C^\infty$ function of $Q$ too.
\end{proof}

Next we turn to the $C^1$ regularity of $J$ on its entire domain. Similarly to before, the main ingredient will be the regularity of the function $G:\text{dom}(J)\times\sm3 \to \sm3$ defined by 
\begin{equation}G(Q,\lambda,\eta)=Q-\left(\int_{E_Q}\exp(\lambda p\cdot p)(Qp\cdot p-\eta)\,dp\right)^{-1}\int_{E_Q}\exp(\lambda p\cdot p)(Qp\cdot p-\eta)\left(p\otimes p-\frac{1}{3}I\right)\,dp.\end{equation}
 
Due to the dependence of the domain of integration on $Q$, it is less clear how regular $G$ is as a function of $Q$. Heuristically, the integrand vanishing on $\partial E_Q$ avoids difficulties up to $C^1$ regularity. We proceed by showing that functions $g:\text{dom}(J)\times\sm3 \to V$, where $V$ is a real vector space, given by
\begin{equation}g(Q,\lambda,\eta)=\int_{E_Q}h(Q,\lambda,p)(Qp\cdot p-\eta)\,dp,\end{equation}
are $C^1$ under appropriate assumptions on $h$. Rather than turn to a proof based in differential geometry, the method we will use to demonstrate this is to consider instead approximations 
\begin{equation}g_\epsilon(Q,\lambda)=\int_{\s2}h(Q,\lambda,p)\varphi_\epsilon(Qp\cdot p-\eta)(Qp\cdot p-\eta)\,dp,\end{equation}
where $\varphi_\epsilon(x) \to \max(x,0)$, $\varphi_\epsilon'(x)\to \chi_{(0,+\infty)}(x)$ and $\varphi_\epsilon \in C^1(\mathbb{R},(0,+\infty))$. The challenge is that the convergence of $\varphi_\epsilon$ cannot be in $C^1$ norm. We will have to permit the convergence to be non-uniform at $x=0$, and then show that this is not problematic since the size of the set where $Qp\cdot p\approx \eta$ can be controlled in a uniform way. More precisely, it will be shown that for all compact $K \subset \{(Q,\eta)\in\mathcal{Q}\times\mathbb{R}:|Q|^2>\eta\}$, 
\begin{equation}\label{eqSetControl}
\limsup\limits_{t \to 0^+}\sup\limits_{(Q,\eta) \in K} \mathcal{H}^2(\{p\in\mathbb{S}^2:|Qp\cdot p-\eta |<t\})=0,
\end{equation}
and that this allows us to prove that $g_\epsilon$ has a limit with the $C^1_{\text{loc}}$ topology. The limit is then shown to be $g$ as expected, providing the necessary result. 

\begin{lemma}
\label{lemmaUniformlySmallMeasure}
Let $K\subset \{(Q,\eta) \in \mathcal{Q}\times \mathbb{R} : |Q|^2>\eta\}$ be compact. Then \begin{equation}\limsup\limits_{t \to 0^+}\sup\limits_{(Q,\eta) \in K} \mathcal{H}^2(\{p\in\mathbb{S}^2:|Qp\cdot p-\eta |<t\})=0.\end{equation}
\end{lemma}
\begin{proof}
Take $t_j \to 0$ and $(Q_j,\eta_j) \in K$ so that 
\begin{equation}\lim\limits_{j \to \infty} \mathcal{H}^2(\{p\in\mathbb{S}^2:|Q_jp\cdot p-\eta_j |<t_j\})=\limsup\limits_{t \to 0^+}\sup\limits_{Q \in K} \mathcal{H}^2(\{p\in\mathbb{S}^2:|Qp\cdot p-\eta |<t\}).\end{equation}
Take a subsequence, not relabelled, so that $t_j \searrow t$, $(Q_j,\eta_j) \to (Q,\eta) \in K$, and $\sup\limits_{p \in \mathbb{S}^2}|(Q_j-Q)p\cdot p+\eta_j-\eta|\searrow 0$. Note that $(Q,\eta)\neq (0,0)$ since $K$ is compact. Let $p \in \{q \in \s2 : |Q_j p\cdot p-\eta_j|<t_j\}$. Then 
\begin{equation}
\begin{split}
t_j >&| Q_j p\cdot p-\eta_j|\\
=& |Q_j p\cdot p-\eta+(Q-Q_j)p\cdot p +(\eta-\eta_j)|\\
\geq &\left||Qp\cdot p-\eta|-|(Q-Q_j)p\cdot p+\eta-\eta_j|\right|
\end{split}
\end{equation}
so that $\delta_j = t_j+\sup\limits_{q \in \mathbb{S}^2}|(Q_j-Q)q\cdot q+\eta_j-\eta|\geq |Qp\cdot p-\eta|$ and $\delta_j\searrow 0$. Therefore using that the sets $\{ p \in \mathbb{S}^2:|Qp\cdot p-\eta|<\delta_j\}$ are nested in $j$,
\begin{equation}
\begin{split}
\limsup\limits_{t \to 0^+}\sup\limits_{(Q,\eta) \in K} \mathcal{H}^2(\{p\in\mathbb{S}^2:|Qp\cdot p-\eta |<t\})=&\lim\limits_{j \to \infty} \mathcal{H}^2(\{p\in\mathbb{S}^2:|Q_jp\cdot p-\eta_j |<t_j\})\\
\leq & \lim\limits_{j \to \infty}\mathcal{H}^2\left(\{p \in \mathbb{S}^2: |Qp\cdot p-\eta|<\delta_j\}\right)\\
=&\mathcal{H}^2\left(\bigcap\limits_{j =1}^\infty \{p\in\mathbb{S}^2 : |Qp\cdot p-\eta|<\delta_j\}\right)\\
=& \mathcal{H}^2\left(\{p\in\mathbb{S}^2:Qp\cdot p=\eta\}\right)=0,
\end{split}
\end{equation}
since $(Q,\eta)\neq (0,0)$ using the pseudo-Haar condition.
\end{proof}

Let $\varphi_\epsilon :\mathbb{R}\to\mathbb{R}$ be so that $\varphi_\epsilon(x)\to \max(x,0)=\varphi_0(x)$ in $L^\infty_{\text{loc}}(\mathbb{R})$, $\varphi_\epsilon'(x)\to \chi_{(0,\infty)}(x)=\varphi_0(x)$ in $L^\infty_{\text{loc}}(\mathbb{R}\setminus\{0\})$, and $||\varphi'_\epsilon||_{K',\infty}$ is bounded independently of $\epsilon>0$ on compact subsets $K'$ of $\mathbb{R}$. The example to have in mind is 
\begin{equation}\varphi_\epsilon(x)=\epsilon\int_{-\infty}^\frac{x}{\epsilon} \frac{\text{erf}(x)+1}{2}\,dx.\end{equation} Let $\mathcal{A}=\{(Q,\lambda,\eta) \in \mathcal{Q} \times \sm3 \times \mathbb{R} : |Q|^2>\eta\}$. 

\begin{proposition}\label{propFirstOrderContinuous}
Let $h :\mathcal{Q}\times \sm3 \times \mathbb{S}^2 \to \mathbb{R}^k$ be continuous. Then for all $K \subset\subset \mathcal{A}$, 
\begin{equation}\int_{\mathbb{S}^2}h(Q,\lambda,p)\varphi_\epsilon(Qp\cdot p-\eta)\,dp\to \int_{E_Q}h(Q,\lambda,p)(Qp\cdot p-\eta)\,dp\end{equation}
uniformly on $K$.
\end{proposition}
\begin{proof}
The result is immediate using that $h$ is continuous so can be bounded independently of $Q,\lambda$ and $p$, and that $\varphi_\epsilon \to \varphi_0$ uniformly.  
\end{proof}

\begin{proposition}\label{propSecondOrderContinuous}
Let $h :\mathcal{Q}\times \sm3 \times \mathbb{S}^2 \to \mathbb{R}^k$. Then for all $K \subset\subset \mathcal{A}$, 
\begin{equation}\int_{\mathbb{S}^2}h(Q,\lambda,p)\varphi_\epsilon'(Qp\cdot p-\eta)\,dp\to \int_{E_Q}h(Q,\lambda,p)\,dp\end{equation}
uniformly on $K$.
\end{proposition}
\begin{proof}
Let $M^+_{Q,\eta,\delta}=\{p \in \mathbb{S}^2: |Qp\cdot p-\eta|\geq\delta\}$ and $M^-_{Q,\eta,\delta}$ denote its complement. Then 
\begin{equation}
\begin{split}
&\left|\int_{\mathbb{S}^2}h(Q,\lambda,p)\varphi_\epsilon'(Qp\cdot p-\eta)\,dp- \int_{E_Q}h(Q,\lambda,p)\,dp\right|\\
=&\left|\int_{\mathbb{S}^2}h(Q,\lambda,p)\varphi_\epsilon'(Qp\cdot p-\eta)\,dp-\varphi_0'(Qp\cdot p-\eta)h(Q,\lambda,p)\,dp\right|\\
\leq & \left|\int_{M^+_{Q,\eta,\delta}}h(Q,\lambda,p)(\varphi_\epsilon'(Qp\cdot p-\eta)-\varphi_0'(Qp\cdot p-\eta)\,dp\right|\\
&+\left|\int_{M^-_{Q,\eta,\delta}}h(Q,\lambda,p)(\varphi_\epsilon'(Qp\cdot p-\eta)-\varphi_0'(Qp\cdot p-\eta)\,dp\right|.
\end{split}
\end{equation}
Since $h$ is continuous and $(Q,\lambda,\eta) \in K$, the $L^\infty$ norm, denoted $||h||_{K,\infty}$, can be pulled out, to give 
\begin{equation}
\begin{split}
&\left|\int_{\mathbb{S}^2}h(Q,\lambda,p)\varphi_\epsilon'(Qp\cdot p-\eta)\,dp- \int_{E_Q}h(Q,\lambda,p)\,dp\right|\\
\leq & ||h||_{K,\infty}\left(\int_{M^+_{Q,\eta,\delta}}|\varphi_\epsilon'(Qp\cdot p-\eta)-\varphi_0'(Qp\cdot p-\eta)|\,dp\right.\\
&+\left.\int_{M^-_{Q,\eta,\delta}}|\varphi_\epsilon'(Qp\cdot p-\eta)-\varphi_0'(Qp\cdot p-\eta)|\,dp\right)\\
\leq & ||h||_{K,\infty}\left(4\pi||\varphi'_\epsilon-\varphi'_0||_{[-L,\delta]\cup[\delta,L],\infty}+\mathcal{H}^2(M^-_{Q,\eta,\delta})||\varphi'_\epsilon-\varphi'_0||_{(-\delta,\delta),\infty}\right)
\end{split}
\end{equation}
for sufficiently large $L>0$. Since $\varphi'_\epsilon \to \varphi'_0$ uniformly on compact sets not containing the origin by assumption, the left-hand term tends to zero. By \Cref{lemmaUniformlySmallMeasure} $\mathcal{H}^2(M^-_{Q,\eta,\delta})$ tends to zero and $||\varphi'_\epsilon-\varphi'_0||_{(-\delta,\delta),\infty}$ remains bounded, so the right-hand term tends to zero also. 
\end{proof}

\begin{proposition}\label{propLimitsGEpsilon}
Let $h :\mathcal{Q}\times \sm3 \times \mathbb{S}^2 \to \mathbb{R}^k$ be $C^1$ in its first two variables, with all derivatives continuous on the entire domain. Let $g_\epsilon : \mathcal{A}\to \mathbb{R}^k$ be given by 
\begin{equation}g_\epsilon(Q,\lambda,\eta)=\int_{\mathbb{S}^2}\varphi_\epsilon(Qp\cdot p-\eta)h(Q,\lambda,p)\,dp.\end{equation}
Then for all $K \subset \subset \mathcal{A}$, 
\begin{equation}
\begin{split}
g_\epsilon(Q,\lambda,\eta)\overset{\text{uni.}}{\to}&\int_{E_Q} (Qp\cdot p-\eta)h(Q,\lambda,p)\,dp\\
\frac{\partial g_\epsilon}{\partial Q}(Q,\lambda,\eta)\overset{\text{uni.}}{\to}& \int_{E_Q}(Qp\cdot p-\eta)\frac{\partial h}{\partial Q}(Q,\lambda,p)+h(Q,\lambda,p)\left(p\otimes p-\frac{1}{3}I\right)\,dp\\
\frac{\partial g_\epsilon}{\partial \lambda} (Q,\lambda,\eta)\overset{\text{uni.}}{\to}& \int_{E_Q}(Qp\cdot p-\eta)\frac{\partial h}{\partial \lambda}(Q,\lambda,p)\\
\frac{\partial g_\epsilon}{\partial \eta} (Q,\lambda,\eta)\overset{\text{uni.}}{\to}&- \int_{E_Q}h(Q,\lambda,p)\,dp
\end{split}
\end{equation}
\end{proposition}
\begin{proof}
Since $\varphi_\epsilon$ and $h$ are $C^1$ we can exchange derivatives in $g_\epsilon$ with ease. We show only the result for $\frac{\partial g_\epsilon}{\partial Q}$ as an example, with the rest following by the same method.
\begin{equation}
\begin{split}
\frac{\partial g_\epsilon}{\partial Q}(Q,\lambda,\eta)=&\frac{\partial}{\partial Q}\int_{\mathbb{S}^2}h(Q,\lambda,p)\varphi_\epsilon(Qp\cdot p-\eta)\,dp\\
=&\int_{\mathbb{S}^2}\frac{\partial h}{\partial Q}(Q,\lambda,p)\varphi_\epsilon(Qp\cdot p-\eta)+h(Q,\lambda,p)\varphi'_\epsilon(Qp\cdot p-\eta)\left(p\otimes p-\frac{1}{3}I\right)\,dp.
\end{split}
\end{equation}
Using \Cref{propFirstOrderContinuous} for the left-hand term and \Cref{propSecondOrderContinuous} for the right-hand term, on $K$ this converges uniformly to 
\begin{equation}\int_{E_Q} \frac{\partial h}{\partial Q}(Q,\lambda,p)(Qp\cdot p-\eta)\,dp+h(Q,\lambda,p)\left(p\otimes p-\frac{1}{3}I\right)\,dp.\end{equation}
\end{proof}

\begin{corollary}\label{propApproxC1}
Let $h :\mathcal{Q}\times \sm3 \times \mathbb{S}^2\to \mathbb{R}^k$, $C^1$ in its first two variables with all derivatives continuous. Then $g:\mathcal{A}\to\mathbb{R}^k$ given by 
\begin{equation}g(Q,\lambda,\eta)=\int_{E_Q}(Qp\cdot p-\eta)h(Q,\lambda,p)\,dp\end{equation}
is $C^1$, with 
\begin{equation}
\begin{split}
\frac{\partial g}{\partial Q}(Q,\lambda,\eta)=& \int_{E_Q}(Qp\cdot p-\eta)\frac{\partial h}{\partial Q}(Q,\lambda,p)+h(Q,\lambda,p)\left(p\otimes p-\frac{1}{3}I\right)\,dp\\
\frac{\partial g}{\partial \lambda} (Q,\lambda,\eta)=& \int_{E_Q}(Qp\cdot p-\eta)\frac{\partial h}{\partial \lambda}(Q,\lambda,p)\\
\frac{\partial g}{\partial \eta} (Q,\lambda,\eta)=&- \int_{E_Q}h(Q,\lambda,p)\,dp
\end{split}
\end{equation}

\end{corollary}
\begin{proof}
Using \Cref{propLimitsGEpsilon}, we have that on all compact subsets of $\mathcal{A}$, $g_\epsilon$ is a Cauchy sequence in $C^1$, and $g_\epsilon \to g$ in $C^0_{\text{loc}}$, therefore $g \in C^1$, and its derivatives are given by the locally uniform limits of the derivatives of $g$. 
\end{proof}

\begin{proposition}\label{propJC1}
$\Lambda$ is a $C^1$ function of $(Q,\eta)$ on $\{(Q,\eta)\in\mathcal{Q}\times\mathbb{R}:|Q|^2>\eta\}$.
\end{proposition}
\begin{proof}
For notational brevity, let $\Lambda=\Lambda(Q,\eta)$ when unambiguous. For each $Q \in\{(Q,\eta)\in\mathcal{Q}\times\mathbb{R}:|Q|^2>\eta\}$, $\Lambda(Q,\eta)$ is uniquely determined, so being globally ill defined is not an issue. The argument will only be needed to show that the map is $C^1$. To see this, note that 
\begin{equation}G(Q,\Lambda,\eta)=Q-\frac{1}{\int_{E_Q}\exp(\Lambda p\cdot p)(Qp\cdot p-\eta)\,dp}\int_{E_Q}\exp(\Lambda p\cdot p)(Qp\cdot p-\eta)\left(p\otimes p-\frac{1}{3}I\right)\,dp=0\end{equation}
for all $Q,\Lambda(Q,\eta),\eta$. In particular, note that both terms in the quotient are $C^1$ in $(Q,\Lambda,\eta)$ by \Cref{propApproxC1}, so that $G$ is $C^1$. The invertibility of $\frac{\partial G}{\partial \Lambda}$ comes from \Cref{eqVariance} at $(Q,\Lambda,\eta)$, with $f_Q(p)=\frac{1}{Z}\exp(\Lambda p\cdot p)\max(Qp\cdot p-\eta,0)$ as before. If $A \in \sm3$ then 
\begin{equation}
\begin{split}
\frac{\partial G}{\partial \Lambda}(Q,\lambda)A\cdot A=&\int_{\mathbb{S}^2} (Ap\cdot p)^2 f_Q(p)\,dp-\left(\int_{\mathbb{S}^2}Ap\cdot p f_Q(p)\,dp\right)^2,
\end{split}
\end{equation}
which by Cauchy-Schwarz is positive unless $Ap\cdot p=0$ on $\supp(f)$, however this implies that $A=0$ since the linearly independent components of $p\otimes p-\frac{1}{3}I$ and the constant function are pseudo-Haar on the sphere with respect to $\mathcal{H}^2$, and $f$ is absolutely continuous with respect to $\mathcal{H}^2$ (see \Cref{propAdmissibleQ}). Therefore $\Lambda$ is a $C^1$ function of $(Q,\eta)$.
\end{proof}

\begin{corollary}\label{corollaryfContinuous}
The map from $Q$ to its optimal energy distribution is continuous with respect to the $W^{1,q}(\s2)$ topology for all $q \in [1,\infty)$, and by extension with respect to $C^\alpha(\s2)$ for $\alpha <1$ and $L^q(\s2)$ for $q \in [1,\infty]$. 
\end{corollary}

\begin{proof}
First we show the continuity in $W^{1,q}(\mathbb{S}^2)$, from which the other results will follow. Let $Q \in \mathcal{Q}\cap\text{dom}(J)$. $f_Q$ is $C^1$ everywhere except where $Qp\cdot p=\eta$, which is a set of zero measure, so it certainly admits a weak derivative. Then using $\nabla_{\mathbb{S}^2}$ to denote the (weak) gradient operator on $\mathbb{S}^2$, 
\begin{equation}\label{equationDerivativeFQ}
\nabla_{\mathbb{S}^2}f_Q(p)=\frac{1}{Z}\exp(\Lambda p\cdot p)\left(\max(Qp\cdot p-\eta,))\Lambda_{\alpha\beta}+\chi_{(0,\infty)}(Qp\cdot p-\eta)Q_{\alpha\beta}\right)\nabla_{\mathbb{S}^2}p_\alpha p_\beta,
\end{equation}
with summation over $\alpha,\beta$. Now if $Q_i \to Q$, and $f_{Q_i}(p)=\frac{1}{Z_i}\exp(\Lambda_i p\cdot p)\max(Q_ip\cdot p-\eta,0)$, then we see that all terms in \Cref{equationDerivativeFQ} converge uniformly, with the exception of $\chi_{(0,\infty)}(Q_ip\cdot p-\eta)$. It therefore suffices to show that if $Q_i \to Q$, then $\chi_{(0,\infty)}(Q_ip\cdot p-\eta)$ converges to $\chi_{(0,\infty)}(Qp\cdot p-\eta)$ in $L^q(\mathbb{S}^2)$. Since these functions only admit values $0,1$ however, it suffices to show that the convergence holds in $L^1(\mathbb{S}^2)$. However by taking $h(Q,\lambda,p)=1$ in \Cref{propApproxC1}, we see that this holds since the function 
\begin{equation}
Q\mapsto \int_{E_Q} \,dp
\end{equation}
is continuous. Therefore if $Q_i \to Q$, then $f_{Q_i}\to f_{Q}$ in $W^{1,q}(\mathbb{S}^2)$ for all $q \in [1,\infty)$, and in particular the convergence also holds in $C^\alpha(\s2)$ for all $\alpha<1$, and in $L^q(\s2)$ for all $q \in [1,\infty]$. 
\end{proof}

\begin{proposition}\label{propJDerivatives}
$J$ is $C^1$ on its domain, with its derivatives given by 
\begin{equation}
\begin{split}
\frac{\partial J}{\partial Q}=&\Lambda(Q,\eta) - \int_{E_Q}\frac{f_Q(p)}{Qp\cdot p-\eta}\left(p\otimes p-\frac{1}{3}I\right)\,dp,\\
\frac{\partial J}{\partial \eta}=&-\int_{E_Q}\frac{f_Q(p)}{Qp\cdot p-\eta}\,dp,
\end{split}
\end{equation}
where $f_Q$ is the optimal energy distribution for $Q$.
\end{proposition}
\begin{proof}
Recalling that $J(Q,\eta)=F(Q,\Lambda(Q,\eta),\eta)$, with $\Lambda$ a $C^1$ function by \Cref{propJC1} and $F$ $C^1$ from \Cref{propApproxC1} gives that $J$ is $C^1$. Since it is known that $J$ and $\Lambda$ are $C^1$ functions of $Q$, it is straightforward to differentiate the expression for $J$. For brevity denote 
\begin{equation}Z=\int_{E_Q}\exp(\Lambda(Q) p\cdot p)(Qp\cdot p-\eta)\,dp.\end{equation}
Then the derivatives of $J$ can be found as
\begin{equation}
\begin{split}
\frac{\partial J}{\partial Q}=&\frac{\partial}{\partial Q} \left(\Lambda\cdot Q - \ln \left(\int_{E_Q}\exp(\Lambda p\cdot p)(Qp\cdot p-\eta)\,dp\right)\right)\\
=&\Lambda -\frac{1}{Z}\int_{E_Q}\exp(\Lambda p\cdot p)\left(p\otimes p-\frac{1}{3}I\right)\,dp+\frac{\partial \Lambda}{\partial Q}\cdot\frac{\partial}{\partial \Lambda} \left(Q\cdot \Lambda-\int_{E_Q}\exp(\Lambda p\cdot p)(Qp\cdot p-\eta)\,dp\right)\\
=&\Lambda - \int_{E_Q}\frac{f_Q(p)}{Qp\cdot p-\eta}\left(p\otimes p-\frac{1}{3}I\right)\,dp.
\end{split}
\end{equation}
Note that the derivative with respect to $\Lambda$ vanishes by the dual optimality condition. By the same argument,
\begin{equation}
\begin{split}
\frac{\partial J}{\partial \eta}=&  \frac{\partial \Lambda}{\partial \eta}\cdot \left(Q-\frac{1}{Z}\int_{E_Q}\exp(\Lambda p\cdot p-\eta)(Qp\cdot p-\eta)\left( p\otimes p-\frac{1}{3}I\right)\,dp\right)+\frac{1}{Z}\int_{E_Q}\exp(\Lambda p\cdot p)\,dp\\
=& \frac{\partial \Lambda}{\partial \eta} \left(Q-\int_{E_Q}f_Q(p)\left(p\otimes p-\frac{1}{3}I\right)\,dp\right)+\int_{E_Q}\frac{f_Q(p)}{Qp\cdot p-\eta}\,dp\\
=&\int_{E_Q}\frac{f_Q(p)}{Qp\cdot p-\eta}\,dp.
\end{split}
\end{equation}
\end{proof}

\begin{theorem}\label{theoremGlobalMinsEquivalent}
Let $f \in \ps2$ be a global minimiser for $\mathcal{F}$. Then $f$ satisfies the Euler-Lagrange equation given by 
\begin{equation}
\begin{split}
f(p)=&\frac{1}{Z}\exp(\Lambda p\cdot p)\max(Qp\cdot p-\eta,0),\\
Q=& \int_{\mathbb{S}^2}f(p)\left(p\otimes p-\frac{1}{3}I\right)\,dp,\\
\Lambda =& \int_{\mathbb{S}^2}\frac{f(p)}{Qp\cdot p-\eta}\left(p\otimes p-\frac{1}{3}I\right)\,dp.
\end{split}
\end{equation}
Furthermore, $Q,\Lambda$ solve the saddle point problem 
\begin{equation}\min\limits_{Q_0\in\text{dom}(J)}\max\limits_{\lambda\in\sm3}Q_0\cdot \lambda-\ln\left(\int_{\mathbb{S}^2}\exp(\lambda p\cdot p)\max(Q_0 p\cdot p-\eta,0)\,dp\right),\end{equation}
and $Q$ satisfies $J(Q)=\min\limits_{Q_0 \in \mathcal{Q}}J(Q_0)$.
\end{theorem}
\begin{proof}
From the minimisation decomposition, $\min\limits_{f \in \ps2}\mathcal{F}(f)=\min\limits_{Q \in \mathcal{Q}}\left(\min\limits_{f \in A(Q)}\mathcal{F}(f)\right)=\min\limits_{Q \in \mathcal{Q}}J(Q)$, first we minimise $J$. Since $\text{dom}(J)$ is open, and $J$ is $C^1$ on its domain, this implies that at the global minimiser, $\frac{\partial J}{\partial Q}(Q)=0$. Therefore 
\begin{equation}\Lambda(Q)-\int_{\mathbb{S}^2}\frac{f_{Q}(p)}{Qp\cdot p-\eta}\left(p\otimes p-\frac{1}{3}I\right)\,dp=0\end{equation}
is the critical point condition for a global minimiser $Q$ of $J$. If $Q$ is a minimiser of $J$, then $f_{Q}$ must also be a global minimiser of $\mathcal{F}$, and the unique minimiser with corresponding Q-tensor $Q$ by uniqueness of solutions to the auxiliary problem. In particular, this means that $f_{Q}$ can be written in the form given in the statement. The saddle-point representation is a direct consequence of \Cref{propFormOfJ}.
\end{proof}

\begin{remark}
The saddle-point representation of the global minimisation problem is at face-value a 10-dimensional problem. However, from \Cref{propQLambdaEigenvectors}, it is known that all eigenvectors of the minimising $Q$ are eigenvectors of $\Lambda(Q)$. Therefore using the frame indifference of the energy, it is possible to fix both into the same diagonal frame, leaving only two degrees of freedom each, so that the problem is only four-dimensional. Furthermore, using a similar argument to \Cref{corollaryfContinuous}, if $(Q_0,\lambda_0)\approx (Q^*,\Lambda(Q)^*))$, where $Q^*$ is a true global minimiser of $J$, then $f_0(p)=\frac{1}{Z}\exp(\lambda_0p\cdot p)\max(Q_0p\cdot p-\eta,0)$ is close in $L^\infty(\s2)$ to a true global minimiser of $\mathcal{F}$, providing a degree of stability in the approximation. Finally it should be noted that the necessary condition for minimisers given in \Cref{theoremGlobalMinsEquivalent} differs only from the vanishing variation condition on $\mathcal{P}_+(\s2)$ that can be obtained from \Cref{propIsotropicGeneral} in as far as the support is unknown. This may suggest that the method presented in this work could be equivalent to considering more careful variations, such as solving $\left.\frac{d}{dt}\mathcal{F}(f_t)\right|_{t=0}=0$ for $f:(-\epsilon,\epsilon) \to \mathcal{P}(\s2)$ smoothly varying in $\text{dom}(\mf)$. This will be addressed more thoroughly in \Cref{propSmoothCurveVariation}.
\end{remark}

We can establish the behaviour of systems as they approach the saturated regime, proving that they achieve perfect order in the appropriate limit. The result requires two ingredients, firstly that we can bound the support of global minimisers onto some small set, and secondly that solutions have even symmetry. First we include a necessary lemma.

\begin{lemma}\label{lemmaLimitingQ}
Let $Q_j \in \mathcal{Q}$, $|Q_j|^2 \to \frac{2}{3}$. Then there exists rotations $R_j \in \text{SO}(3)$ such that $R_jQ_jR_j^T \to \left(e_1\otimes e_1-\frac{1}{3}I\right)$. In particular, as $\eta \to \frac{2}{3}$ from below, the corresponding minimisers $Q_j$ of $J$ satisfy $R_jQ_jR_j^T \to e_1\otimes e_1-\frac{1}{3}I$.
\end{lemma}
\begin{proof}
First we show that $\frac{2}{3}=\max\limits_{Q \in \overline{\mathcal{Q}}}|Q|^2$, and that if $|Q|^2 =\frac{2}{3}$ for $Q \in \overline{\mathcal{Q}}$, then $Q=e\otimes e-\frac{1}{3}$ for some $e \in \s2$. Since $\mathcal{Q}$ is a bounded set, there exists at least one maximiser of $|Q|^2$ on $\overline{\mathcal{Q}}$. The maximum cannot be in the interior, since if $Q \in \mathcal{Q}$ then $(1+\epsilon)Q \in \mathcal{Q}$ for small $\epsilon>0$. Since $Q \in \partial\mathcal{Q}$, at least one eigenvalue of $Q$ must be $-\frac{1}{3}$. This means we can write $Q$ as a diagonal matrix in its eigenframe, $Q=D\left(-\frac{1}{3},a,\frac{1}{3}-a\right)$ for some $a \in \left[-\frac{1}{3},\frac{2}{3}\right]$.  In this case, $|Q|^2= 2a^2-\frac{2a}{3}+\frac{1}{9}$. This is a positive quadratic and therefore the maximum must be on the boundary. Either choice $a=-\frac{1}{3},\frac{2}{3}$ gives the same result up to permuting the eigenvectors that $Q=D\left(-\frac{1}{3},-\frac{1}{3},\frac{2}{3}\right)$, with $|Q|^2=\frac{2}{3}$.
Now take $|Q_j|^2 \to \frac{2}{3}$ with $Q_j \in \mathcal{Q}$. Take $R_j \in \text{SO}(3)$ to be rotations so that $R_jQ_jR_j^Te_1 = v_{\max}(Q_j)e_1$. Assume that $R_jQ_jR_j^T$ does not have limit $\left(e_1\otimes e_1-\frac{1}{3}I\right)$, then there would be a subsequence (not relabelled) such that $\left|R_jQ_jR_j^T-e\otimes e-\frac{1}{3}I\right| $ is bounded away from zero, and also $R_jQ_jR_j^T \to Q^*$ for some $Q^*$ by compactness. It must then hold that $|Q^*|^2=\frac{2}{3}$ and $Q^* \in \overline{\mathcal{Q}}$, so $Q^*=e\otimes e-\frac{1}{3}$ for some $ e \in \s2$. Finally, $v_{\max}(Q_j)e_1=R_jQ_jR_j^Te_1\to Q^*e_1 = v_{\max}(Q^*)e_1$ by continuity of the largest eigenvalue, so this implies that $Q^*=e_1\otimes e_1-\frac{1}{3}I$.
The conclusion that global minimisers must approach $e_1\otimes e_1-\frac{1}{3}$ then holds since $\eta_j < |Q_j|^2<\frac{2}{3}$, so as $\eta_j \nearrow \frac{2}{3}$, $|Q_j|^2 \to \frac{2}{3}$. 
\end{proof}

\begin{proposition}\label{propLimitingF}
Let $\eta_j \to \frac{2}{3}$ from below, and $f_j$ be a corresponding global minimiser of $\mathcal{F}$. Given $R \in \text{SO}(3)$ let $[Rf]\in\ps2$ be given by $[Rf](p)=f(Rp)$. Then there exists rotations $R_j$ such that, $[R_jf_j] \wsl \frac{1}{2}\left(\delta_{e_1}+\delta_{-e_1}\right)$ in $C(\s2)^*$.
\end{proposition}
\begin{proof}
Let $Q_j$ denote the Q-tensor of $f_j$. By the \Cref{lemmaLimitingQ}, we take $R_j$ so that $R_jQ_jR_j^T \to e\otimes e-\frac{1}{3}I=Q$ for a given $e\in\s2$. Using that $R_jQ_jR_j^T p\cdot p - \eta_j$ converges uniformly on $\mathbb{S}^2$ to $Qp\cdot p-\frac{2}{3}$, we have that for a given $\epsilon >0$ and sufficiently large $j$, $\text{supp}([R_jf_j])\subset \left\{p \in \s2 : Qp\cdot p-\frac{2}{3}>-\epsilon\right\}=U_\epsilon$. Let $g \in C(\mathbb{S}^2)$. Then using that $f_j(p)=f_j(-p)$ for all $p \in \s2$, this gives that for sufficiently large $j$,
\begin{equation}
\begin{split}
\int_{\s2}[R_jf_j](p)g(p)\,dp=&\frac{1}{2}\int_{\s2}f_j(R_jp)\left(g(p)+g(-p)\right)\,dp\\
=&\int_{U_\epsilon}f_j(R_jp)\frac{g(p)+g(-p)}{2}\,dp\\
\in & \left[ \min\limits_{p \in U_\epsilon}\frac{g(p)+g(-p)}{2},\max\limits_{p \in U_\epsilon}\frac{g(p)+g(-p)}{2}\right].
\end{split}
\end{equation}
This implies that the limit inferior and limit superior of $\int_{\s2}[R_jf_j](p)g(p)\,dp$ as $j \to +\infty$ must lie in this interval also. However, $\epsilon>0$ was arbitrary. Using that $g$ is continuous and $U_\epsilon$ are nested, open neighbourhoods of $\{-e_1,e_1\}$ with $\bigcap\limits_{\epsilon>0}U_\epsilon=\{-e_1,e_1\}$, it holds that \begin{equation}\min\limits_{p \in U_\epsilon}\frac{g(p)+g(-p)}{2}\nearrow \min\limits_{p \in \{-e_1,e_1\}}\frac{g(p)+g(-p)}{2}=\frac{g(e_1)+g(-e_1)}{2}.\end{equation}
The same argument gives that $\max\limits_{p \in U_\epsilon}\frac{g(p)+g(-p)}{2}\searrow \frac{g(e_1)+g(-e_1)}{2}$. Therefore a squeezing argument gives that 
\begin{equation}\lim\limits_{j \to \infty} \int_{\s2}[R_jf_j](p)g(p)\,dp = \frac{g(e_1)+g(-e_1)}{2}\end{equation}
as required.

\end{proof}

\begin{proposition}
If $\eta<-\frac{1}{3}$, then the isotropic state is the unique global minimiser of $\mathcal{F}$.
\end{proposition}
\begin{proof}
Assume $\eta<-\frac{1}{3}$. For $Q \in  \mathcal{Q}$, we have that the optimal energy distribution for $Q$ is given by $f_Q(p)=\frac{1}{Z}\exp(\Lambda(Q)p\cdot p)(Qp\cdot p-\eta)$, and in particular it is bounded away from zero and $+\infty$. Since all global minimisers of $\mf$ must be the optimal energy distribution for their Q-tensor, all global minimisers are bounded away from zero and $+\infty$, and the global minimiser satisfies $f^* \in \mathcal{P}_+(\s2)$. By \Cref{propositionIsotropicStateGlobalMin}, the isotropic state is the unique global minimiser of $\mathcal{P}_+(\s2)$, so the isotropic must be the unique global minimiser on $\ps2$.
\end{proof}

While obtaining the full phase diagram analytically appears to be out of reach, we can obtain some further qualitative results on the phase diagram. We will now show that for $\eta$ sufficiently small, the isotropic state is not a global minimiser. The result holds trivially for $\eta\geq 0$ since the isotropic state has infinite energy.

\begin{proposition}\label{propIsotropicNotGlobal}
There exists some $\eta^*<0$ so that for $\eta^*<\eta<0$, the isotropic state is not a global minimiser of $\mathcal{F}$.
\end{proposition}
\begin{proof}
We will show this result using a perturbation argument. For clarity, we will recall the dependence of $\mathcal{F}$ explicitly on $\eta$. We know that for $|Q|^2>\eta$, $\mathcal{F}$ is continuous. Take $Q^* \in \mathcal{Q}\setminus \{0\}$. Then $J(Q^*,0)=M<+\infty$. Since $J$ is continuous, this means that $J(Q^*,\eta)<2M$ for all $|\eta|$ sufficiently small. Since $J$ blows up uniformly at the boundary of its domain and $(0,0)\in\partial\text{dom}(J)$, we know that there exists some ball $(0,0)\in B \subset \mathcal{Q}\times\mathbb{R}$ so that if $(Q,\eta)\in B$, then $J(Q,\eta)>2M$. In particular, if $\{\eta \in\mathbb{R}: \exists Q, (Q,\eta)\in B\}=B'\subset \mathbb{R}$, then $J(0,\eta)>2M$ for all $\eta \in B'$. Therefore $J(Q^*,\eta)<2M<J(0,\eta)$ for all $|\eta|$ sufficiently small. Therefore the isotropic state is not the global minimiser for sufficiently small $|\eta|$.
\end{proof}

\begin{remark}\label{remarkVdW}
It is possible to use this free energy density to provide an analogue of the Van der Waals equation of state for the system. Returning to original units, the free energy density is given by 
\begin{equation}\hat{\mathcal{F}}(f)=k_BT\left(\rho\ln \rho + \rho\int_{\mathbb{S}^2}f(p)\ln f(p)-f(p)\ln \left(1-\rho\left(c-\frac{3}{2}dQp\cdot p\right)\right)\,dp\right).\end{equation}
$\eta$ is related to $\rho$ by $\eta=\frac{2(\rho c-1)}{\rho d}$, with $c,d>0$ material constants satisfying $\frac{d}{3c-2d}>0$ and $\rho>0$ is the number density. This is related to the pressure $P$ via
\begin{equation}P=-\hat{\mathcal{F}}+\rho\frac{\partial\hat{\mathcal{F}}}{\partial \rho}.\end{equation}
There is an issue in that the derivative of the free energy cannot rigorously be taken with respect to number density at non-trivial minimisers due to the support condition. In \cite{zheng2016density} the expression for pressure is obtained non-rigorously. However, by reducing the energy to the macroscopic free-energy function $J$, the energy becomes sufficiently regular for the derivative with respect to number density to be taken (see \Cref{propJDerivatives}), providing consistent results with Zheng {\it et al.}. First, we note that in original units, $J=J(Q,\eta)$ can be written as 
\begin{equation}\min\limits_{f \in A(Q)}\hat{\mathcal{F}}(f)=k_BT\left(\rho J\left(Q,\frac{2(\rho c-1)}{3\rho d}\right)-\rho \ln d\right).\end{equation}
 This then gives the pressure as 
\begin{equation}
\begin{split}
\frac{1}{k_BT}P=&\rho\frac{\partial}{\partial \rho}\left(\rho J\left(Q,\frac{2(\rho c-1)}{3\rho d}\right)-\rho \ln d\right)-\left(\rho J\left(Q,\frac{2(\rho c-1)}{3\rho d}\right)-\rho \ln d\right)\\
=&\rho^2\frac{\partial J}{\partial \eta}\left(Q,\frac{2(\rho c-1)}{3\rho d}\right)\frac{\partial}{\partial \rho}\left(\frac{2(\rho c-1)}{3\rho d}\right)\\
=& \frac{2}{3}\rho^2\int_{E_Q}\frac{f_Q(p)}{Qp\cdot p-\frac{2(\rho c-1)}{3\rho d}}\,dp\left(\frac{\rho cd- d(\rho c-1)}{(\rho d)^2}\right)\\
=& \frac{1}{d} \int_{E_Q}\frac{f_Q(p)}{Qp\cdot p -\frac{2(\rho c-1)}{3\rho d}}\,dp \end{split}
\end{equation}
This then gives the generalised Van der Waals equation of state,
\begin{equation}\label{eqVdW}
P= k_BT\rho\int_{E_Q}\frac{f_Q(p)}{1-\rho\left(c-\frac{3}{2}dQp\cdot p\right)}\,dp
\end{equation}
This then allows us to demonstrate how the pressure blows up at the saturation limit, since returning to dimensionless units and using Jensen's inequality,
\begin{equation}
\begin{split}
\frac{2\rho d}{3k_BT}P=&\int_{E_Q}\frac{f_Q(p)}{Qp\cdot p-\eta}\,dp\\
\geq & \frac{1}{\int_{E_Q}\left(Qp\cdot p-\eta\right)f_Q(p)\,dp}\\
=& \frac{1}{|Q|^2-\eta}\\
\geq & \frac{1}{\frac{2}{3}-\eta}\\
=&\frac{3\rho d}{3-\rho(2d-3c)}\\
=& \left(\frac{3\rho d}{3c-2d}\right)\frac{1}{\frac{3}{2d-3c}-\rho}
\end{split}
\end{equation}
where the saturation limit $\eta \nearrow \frac{2}{3}$ corresponds to $\rho \nearrow \rho_s = \frac{3}{3c-2d}$. This gives a lower bound of the blow up rate of the pressure as $P\geq \frac{C}{\rho_s-\rho}$. In particular, we can consider a dimensionless analogue of pressure, $P^*=\frac{2\rho d}{3k_BT}P$, and see that we have 
\begin{equation}\label{equationBoundsDimensionlessPressure}
P^*\geq\frac{1}{|Q|^2-\eta}\geq \frac{1}{\frac{2}{3}-\eta}.
\end{equation}
\end{remark}

\section{Local minimisers of $\mathcal{F}$}
\label{secLocalMinimisers}
By splitting the global minimisation into two manageable minimisation problems, it was possible to reduce the minimisation of $\mathcal{I}_{\mathcal{P}(X)}$ to a finite dimensional problem. The next natural question is if analogous results can be obtained for local minimisers. In the case of the infinite dimensional problem, in the general case one must take care as to with respect to which topology a local minimiser refers to. In the following, a general framework for establishing equivalence between local minimisers of analogous problems will be presented so that it may be adapted in later sections with ease. 

\begin{definition}\label{definitionsForLocalMins}
Let $(V,||\cdot ||_V)$ be a Banach space, $U\subset V$, $T:V\to\mathbb{R}^k$ be a finite-rank continuous linear operator. Let $\mathcal{Q}=TU$, and $\mathcal{F}:U\to\mathbb{R}\cup\{+\infty\}$ admit a lower bound, be coercive and be lower semicontinuous with respect to the weak topology on $V$. Assume that if $f \in U$ and $\mathcal{F}(f)<+\infty$, then $\mathcal{F}$ is strictly convex on the set where $b=Tf$ is constant. Analogously to before, define $A(b)=\{f \in U : Tu=b\}$ and let $D=\{Tu:\mathcal{F}(f)<+\infty\}$. Define the right inverse $T^{-1}:D\to V$ by 
\begin{equation}T^{-1}(b)=\argmin\limits_{f\in A(b)} \mathcal{F}(f).\end{equation}
Define $U_\mathcal{F}=T^{-1}D$. In the following, assume $U_\mathcal{F}\subset X\subset V$, where $(X,||\cdot ||_X)$ is a Banach space with $||\cdot ||_X$ inducing a topology at least as strong as that induced by $||\cdot ||_V$ when restricted to $X$, and $T^{-1}$ is continuous with respect to the topology induced by $||\cdot ||_X$. Finally, define 
\begin{equation}
J(b)=\min\limits_{f \in A(b)}\mathcal{F}(b).
\end{equation}
\end{definition}

\begin{remark}
Note that $T^{-1}$ is only a right inverse, since $TT^{-1}=\text{Id}$, but $T^{-1}Tf = f$ if and only if $f \in \mathcal{P}_{\mathcal{F}}(\mathbb{S}^2)$. In this work, candidates for $X$ will be the $L^p$ spaces for $1\leq p \leq \infty$ and $C^k$ spaces, and the the standard topology on $V$ will be the $L^1$ topology. Furthermore, the conditions on $\mathcal{F}$ ensure a unique solution $T^{-1}b$. 
\end{remark}

\begin{proposition}\label{propMinsInMinSet}
Assume that $f^* \in U$ is an $X$-local minimiser. Then $f^* \in U_\mathcal{F}$. 
\end{proposition}
\begin{proof}
For the sake of contradiction assume otherwise. Let $f = T^{-1}Tf^*\neq f^*$. Then $Tf=Tf^*$, and $\mathcal{F}(f^*)>\mathcal{F}(f)$. Furthermore, since $\mathcal{F}$ is strictly convex on $A(Tf)$, this means that for $1>\gamma>0$
\begin{equation}
\begin{split}
\mathcal{F}(\gamma f+(1-\gamma)f^*)<&\gamma\mathcal{F}(f)+(1-\gamma)\mathcal{F}(f^*)\\
\leq & \gamma \mathcal{F}(f^*)+(1-\gamma)\mathcal{F}(f^*)=\mathcal{F}(f^*).
\end{split}
\end{equation}
However, $||(\gamma f + (1-\gamma)f^*)-f^*||_X=\gamma||f^*||_X$, so by taking $\gamma \to 0$, this contradicts that $f^*$ is an $X$-local minimiser.
\end{proof}

\begin{proposition}\label{propFMinQ}
Assume that $f^*$ is an $X$-local minimiser of $\mathcal{F}$. Then $b^*=Tf^*$ is a local minimiser of $J$. 
\end{proposition}
\begin{proof}
Assume otherwise for the sake of contradiction. Then there exists $b_j \to b$ with $J(b_j)<J(b^*)$ for all $j$. Then, by the continuity assumption, $f_j=T^{-1}b_j\to f^*=T^{-1}b^*$ in $X$, where the final equality holds because $f \in U_\mathcal{F}$. Then 
\begin{equation}\mathcal{F}(f^*)=J(b^*)>J(b_j)=\mathcal{F}(f_j),\end{equation}
contradicting that $f^*$ is an $X$-local minimiser of $\mathcal{F}$.
\end{proof}

\begin{proposition}\label{propQMinF}
Assume $b^*$ is a local minimiser for $J$. Then $T^{-1}b^*=f^*$ is an $X$-local minimiser for $\mathcal{F}$. 
\end{proposition}
\begin{proof}
Assume otherwise for the sake of contradiction, so that there exists $f_j \to f$ in $X$ with $\mathcal{F}(f_j)<\mathcal{F}(f^*)$ for all $j$. Let $b_j = Tf_j$. In particular, since $f_j \to f$, this implies $b_j \to b$. Furthermore
\begin{equation}J(b^*)=\mathcal{F}(f^*)>\mathcal{F}(f_j)\geq J(b_j),\end{equation}
contradicting that $b^*$ is a local minimiser of $J$. 
\end{proof}

\begin{proposition}\label{corollaryNoLavrentiev}
Let $X,X'$ be any two topologies satisfying the conditions in \Cref{definitionsForLocalMins}. Then $f^*$ is an $X$-local minimiser if and only if it is an $X'$-local minimiser
\end{proposition}
\begin{proof}
The proof is symmetric, so only one direction will be shown. Assume that $f^*$ is an $X$-local minimiser. Then $Tf^*$ is a local minimiser of $J$ by \Cref{propFMinQ}. Therefore $T^{-1}Tf^*$ is an $X'$-local minimiser of $\mathcal{F}$ by \Cref{propQMinF}. Finally, $T^{-1}Tf^*=f^*$ by \Cref{propMinsInMinSet}. 
\end{proof}

\begin{remark}\label{remarkGeneralMinimisation}
The proofs above are in fact far more general than this particular case. Given a real topological vector space $V$, functional $\mathcal{F}:V\to\mathbb{R}\cup\{+\infty\}$, and a continuous linear operator on $V$ with finite dimensional range $T$ so that $\mathcal{F}$ is strictly convex on the subsets $A(Tv)=\{\tilde{v} \in V : T\tilde{v}=Tv\}$, the proofs all carry over exactly the same. In particular, this generalises the results for Onsager-type models as given in \cite{taylor2015maximum}, where $\mathcal{F}:\mathcal{P}(\mathbb{S}^2)\to\mathbb{R}\cup\{+\infty\}$ is 
\begin{equation}\mathcal{F}(f)=\int_{\Omega}f(t)\ln f(t)\,d\mu(t)-\int_\Omega\int_\Omega f(t)f(s)\sum\limits_{i,j=1}^kc_{ij}a_i(t)a_j(s)\,d\mu(t)\,d\mu(s),\end{equation}
under some assumptions on the functions $a_i\in L^\infty(\Omega)$. Similarly, if the appropriate continuity results are provided, these results can be applied to generalisations of the functional given in \cite{zheng2016density} accounting for more general state spaces and excluded volume terms, such as 
\begin{equation}\mathcal{F}(f)=\int_\Omega f(t)\ln f(t) - f(t)\ln \left(\int_\Omega\sum\limits_{i,j=1}^k c_{ij}a_i(t)a_j(s)f(s)\,d\mu(s)-\eta\right)\,d\mu(t).\end{equation}
\end{remark}

\begin{theorem}\label{theoremNoLavrentiev}
If $f \in \mathcal{P}(\s2)$ is a $W^{1,q}$-local minimiser for any $q \in [1,\infty)$, then $f$ is an $L^1$-local minimiser. In particular, local minimisation with respect to $L^q$ ($q \in [1,\infty]$), $C^\alpha$ ($\alpha<1$) or $W^{1,q}$ ($q \in [1,\infty)$) are all equivalent. 
\end{theorem}
\begin{proof}
This follows from \Cref{corollaryNoLavrentiev} and \Cref{corollaryfContinuous}.
\end{proof}

\begin{theorem}[Euler-Lagrange equation for $L^p$ local minimisers]\label{theoremLocalMinsEQuiv}
Assume that $f \in \ps2$ is an $L^p$-local minimiser of $\mathcal{F}$. Then $f$ satisfies the Euler-Lagrange equation
\begin{equation}
\begin{split}
f(p)=&\frac{1}{Z}\exp(\Lambda p\cdot p)\max(Qp\cdot p-\eta,0),\\
Q=& \int_{\mathbb{S}^2}f(p)\left(p\otimes p-\frac{1}{3}I\right)\,dp,\\
\Lambda =& \int_{\mathbb{S}^2}\frac{f(p)}{Qp\cdot p-\eta}\left(p\otimes p-\frac{1}{3}I\right)\,dp.
\end{split}
\end{equation}
Furthermore, $(Q,\Lambda)$ is a critical point of the dual function $F$.
\end{theorem}
\begin{proof}
This is a consequence of \Cref{propQMinF} by the same argument as \Cref{theoremGlobalMinsEquivalent}.
\end{proof}

\begin{corollary}
The uniform state $f_U(p)=\frac{1}{4\pi}$ is an $L^1$-local minimiser for all $\eta<0$. 
\end{corollary}
\begin{proof}
In \Cref{propIsotropicGeneral} and \Cref{corollaryIsotropicDrama} it is shown that $f_U$ is an $L^\infty$ local minimiser if $\eta <0$, so by \Cref{theoremNoLavrentiev} it is an $L^1$ local minimiser also. 
\end{proof}

\begin{corollary}
No $L^1$-local minimisers can be found by solving the equation 
\begin{equation}\left.\frac{d}{dt}\mathcal{F}(f^*+t\phi)\right|_{t=0}=0\end{equation}
for all $\phi \in L^\infty(\supp(f))$ with $\int_{\supp(f)}\phi(p)\,dp=0$, with the exception of the isotropic state when $\eta<0$ and the set $\left\{\frac{1+u}{4\pi}: u \in V\right\}\cap \mathcal{P}(\mathbb{S}^2)$ as given in \Cref{propEta215Drama}
\end{corollary}
\begin{proof}
Let $f^*$ be an $L^1$-local minimiser that is not one of the exceptions given. If $Q^*$ is the corresponding Q-tensor of $f^*$, then $v_{\min}(Q)<\eta$ and $\{p\in\mathbb{S}^2:Qp\cdot p-\eta>0\}$ has positive measure. Take any $\phi \in L^\infty(\supp(f))$ with $\supp\phi$ compactly supported in $\text{int}\supp (f^*)$, and also so that $A=\int_{\supp(f)}\left(p\otimes p-\frac{1}{3}I\right)\phi(p)\,dp \neq 0$. Take any $p_0 \in \mathbb{S}^2$ so that $Qp_0\cdot p_0-\eta=0$ and $Ap_0\cdot p_0\neq 0$, without loss of generality taking $Ap\cdot p<0$.  Take $t>0$. Then $Qp_0\cdot p_0+tAp_0\cdot p_0 -\eta < Qp_0\cdot p_0 -\eta =0$, so there is a neighbourhood $B_t$ of $p_0$ so that $Qp\cdot p+tAp\cdot p-\eta<0$ for all $p \in B_t$. Since $p_0\in\{p \in \mathbb{S}^2:Qp\cdot p=\eta \}= \partial\supp (f^*)$, this means that $B_t\cap \supp(f^*)$ and $B_t\cap \text{int}\supp (f)$ have positive measure for all $t$. Since $\phi$ is supported on the interior of $\text{supp}(f^*)$, this implies that for $t$ sufficiently small, $f^*+t\phi$ is positive on $\text{int}\supp(f^*)$. Therefore the energy has a contribution 
\begin{equation}\mathcal{F}(f+t\phi)\geq c + \int_{B_t\cap \text{int}\supp(f^*)}(f(t)+t\phi )\ln (Qp\cdot p+tAp\cdot p-\eta)\,dp=+\infty.\end{equation}
In particular, since $\mathcal{F}(f+t\phi)=+\infty$ for $t>0$, the limit
\begin{equation}\lim\limits_{t \to 0}\frac{\mathcal{F}(f+t\phi)-\mf(f)}{t}\end{equation}
can be at best infinite, and certainly non-zero
\end{proof}

We can also provide a similar result, which can loosely be interpretted as saying that, the function $\mathcal{F}$ is not locally convex at any non-trivial minimisers. This rules out, for example, finding variational inequalities in local regions of the domain. First we include a lemma.

\begin{lemma}
Let $U_1 \subset \mathbb{S}^2$ be open, and $U_2 \subset\text{SO}(3)$ be open, with $I \in U_2$. Then if $U_2U_1=\{Rp:R\in U_2,\, p \in \mathbb{S}^2\}=U_1$, either $U_1=\emptyset$ or $U_1 =\mathbb{S}^2$.
\end{lemma}
\begin{proof}
First we will show that $U_2U_1=U_1$ is closed. Let $q \in \partial U_2U_1=\partial U_1$. Therefore there exists a sequence $q_j \in U_1$ so that $q_j \to q$. Furthermore, since $U_2$ is open and contains the identity, there exists some $\delta>0$, independent of $p \in \mathbb{S}^2$, so that $B(p,\delta)\subset U_2p$. In particular, $B(q_j,\delta)\subset U_2 q_j \subset U_2U_1=U_1$, and furthemore this gives $B(q,\delta)\subset U_1$. Therefore $q \in U_1$, and in particular $\partial U_1 \subset U_1$. Therefore $U_1$ is closed, and also open by assumption, meaning that since $\mathbb{S}^2$ is connected either $U_1=\emptyset$ or $U_1=\mathbb{S}^2$. 
\end{proof}

\begin{proposition}\label{propNotLocallyConvex}
Let $f^*\in\mathcal{P}(\mathbb{S}^2)$ be an $L^1$-local minimiser of $\mathcal{F}$, which is neither the isotropic state when $\eta<0$ nor in the set $\left\{\frac{1+u}{4\pi}: u \in V\right\}\cap \mathcal{P}(\mathbb{S}^2)$ as given in \Cref{propEta215Drama}. Then for all $\epsilon>0$ there exists some $f \in \text{dom}(\mathcal{F})$ with $||f^*-f||_\infty<\epsilon$ and $\xi \in (0,1)$ so that $\mathcal{F}((1-\xi)f^*+\xi f)=+\infty$. 
\end{proposition}
\begin{proof}
Let $Q^*$ denote the Q-tensor of $f^*$. First we note that since $f^*$ is a non-trivial minimiser, $f^*(p)>0$ if and only if $p \in E_{Q^*}$, and $E_{Q^*} \neq \mathbb{S}^2$. Since $E_{Q^*}$ is open and not $\mathbb{S}^2$, this means that if we let $B_\delta$ denote the ball of radius $\delta$ about the identity in $\text{SO}(3)$, then there exists some $R \in B_\delta$ so that $RE_{Q^*} \neq E_{Q^*}$. Take $\delta$ sufficiently small so that $||f_{RQR^T}-f^*||_\infty<\epsilon$ for all $R \in B_\delta$. Furthermore, since the sets are open, this means that their symmetric difference, $\big(RE_{Q^*}\setminus E_{Q^*}\big) \cup \big(E_{Q^*} \setminus RE_{Q^*}\big)$ has positive measure. This implies at least one term in the union has positive measure, which we take to be $RE_{Q^*}\setminus E_{Q^*}$, with the proof for the alternative following identically. 

Let $Q=RQ^*R^T$. We note that $RE_{Q^*}=E_{RQ^*R^T}=E_Q$ and let $f=f_{Q}$. In this case, we have that $f^\xi(p)=\xi f(p)+(1-\xi)f^*(p)>0$ on $E_{Q^*} \cup E_{Q}$. In order for $f^\xi$ to have finite energy, it is required that $f^\xi(p)$ only be positive on $E_{\xi Q+(1-\xi)Q^*}$ up to a set of measure zero. Therefore, if $\mathcal{H}^2\left((f^\xi)^{-1}(0,\infty)\setminus E_{\xi Q + (1-\xi)Q^*}\right)>0$, $f^\xi$ has infinite energy. This can then be estimated as 
\begin{equation}
\begin{split}
&\mathcal{H}^2\left((f^\xi)^{-1}(0,\infty)\setminus E_{\xi Q + (1-\xi)Q^*}\right)\\
=&\mathcal{H}^2\left(\left(E_Q\cup E_{Q^*}\right)\setminus E_{\xi Q + (1-\xi)Q^*}\right)\\
=& \mathcal{H}^2\left(\left(E_Q\setminus E_{\xi Q + (1-\xi)Q^*}\right)\cup \left(E_{Q^*}\setminus E_{\xi Q + (1-\xi)Q^*}\right)\right)\\
\geq & \mathcal{H}^2\left(E_Q\setminus E_{\xi Q + (1-\xi)Q^*}\right).
\end{split}
\end{equation}
Using the pseudo-Haar condition we have that the limit can be taken and give 
\begin{equation}
\begin{split}
&\mathcal{H}^2\left((f^\xi)^{-1}(0,\infty)\setminus E_{\xi Q + (1-\xi)Q^*}\right)\\
\geq &\lim\limits_{\xi \to 1} \mathcal{H}^2\left(E_Q\setminus E_{\xi Q + (1-\xi)Q^*}\right)\\
=& \mathcal{H}^2(E_Q\setminus E_{Q^*}),
\end{split}
\end{equation}
which was taken to have positive measure. This implies for $\xi$ sufficiently close to $1$, $\mathcal{F}(f^\xi)=+\infty$.
\end{proof}

One might ask the question of how to interpret a critical point of $J$ in terms of the microscopic model. In particular, the numerical studies in \cite{zheng2016density} provide evidence for the existence of critical points of $J$ that are not local minimisers. The inability to take arbitrary $L^\infty$ variations about $f_Q$ when $Q$ is a non-trivial critical point of $J$ means that one cannot easily say in what sense $f_Q$ should be a critical point of $\mathcal{F}$. The next result shows that the critical points of $J$ are in one-to-one correspondence with points $f_0$ where, for a certain family of curves $f_t$ in $\text{dom}(\mf)$, 
\begin{equation}\left.\frac{d}{dt}\mathcal{F}(f_t)\right|_{t=0}=0\end{equation} 
Without the toolkit developed in this work however it is unclear if all local minimisers of $\mathcal{F}$ can be found using such curves, but the results presented here answer the question in the affirmative.  First, we will need a lemma concerning the differentiability of the map $Q \mapsto f_Q$. 

\begin{lemma}\label{lemmaQToFC1}
Let $Q \in \text{dom}(J)$, $p \in \{ q \in \s2 : Qq\cdot q>\eta\}$. For a function $a:\s2 \to \mathbb{R}$, and $f \in \ps2$, define $ [ a (p)]_f=a(p)-\int_{\s2}a(q)f(q)\,dq$. Then 
\begin{equation}\frac{\partial f_{Q}(p)}{\partial Q}= f_Q(p)\left[\sum\limits_{i,j=1}^3\frac{\partial \Lambda_{ij}}{\partial Q}p_ip_j+ \frac{1}{Qp\cdot p-\eta}\left(p\otimes p-\frac{1}{3}I\right)\right]_{f_Q}.\end{equation}
In particular, the map $(Q,p)\mapsto \frac{\partial f_Q(p)}{\partial Q}$ is continuous for $\bigcup\limits_{Q \in \text{dom}(J)}\{Q\}\times E_Q$.
\end{lemma}
\begin{proof}
The result is found by directly differentiating the expression 
\begin{equation}f_Q(p)=\frac{1}{Z}\exp(\Lambda(Q)p\cdot p)(Qp\cdot p-\eta),\end{equation}
on the domain $Qp\cdot p>\eta$, noting that 
\begin{equation}
\begin{split}
Z=&\int_{E_Q}\exp(\Lambda(Q)p\cdot p)(Qp\cdot p-\eta)\,dp\\
\Rightarrow \frac{\partial Z}{\partial Q}=& \int_{E_Q}\frac{\partial}{\partial Q}(\exp(\Lambda(Q)p\cdot p)(Qp\cdot p-\eta))\,dp\\
\Rightarrow \frac{\partial f_Q(p)}{\partial Q}=& \frac{1}{Z}\frac{\partial}{\partial Q}\exp(\Lambda(Q)p\cdot p)(Qp\cdot p-\eta) -\frac{1}{Z}f_Q(p)\int_{E_Q}\frac{\partial}{\partial Q}\exp(\Lambda(Q)p\cdot p)(Qp\cdot p-\eta)\,dp\\
=&f_Q(p)\left[\sum\limits_{i,j=1}^3\frac{\partial \Lambda_{ij}}{\partial Q}p_ip_j+ \frac{1}{Qp\cdot p-\eta}\left(p\otimes p-\frac{1}{3}I\right)\right]_{f_Q}.
\end{split}
\end{equation}
The continuity of this map follows from the explicit representation using that $\Lambda$ is a $C^1$ function of $Q$ and the map $Q \to f_Q$ is continuous with $L^\infty$. 
\end{proof}

\begin{definition}
Let 
\begin{equation}V=\left\{ u \in L^\infty(\s2): \int_{\s2}\left(p\otimes p-\frac{1}{3}I\right)u(p)\,dp=0\,,\, \int_{\s2}u(p)\,dp=0\right\}.\end{equation} 
We say that a map $f_t :(-\epsilon,\epsilon)\to L^1(\s2)$ satisfies assumption (A1) if there exists $Q \in C^1(-\epsilon,\epsilon;\text{dom}(J))$, $u,u_0 \in V$ so that
\begin{enumerate}
\item $f_t=f_{Q_t}+u_0+tu$.
\item $\supp(u)$ and $\supp(u_0)$ are compactly supported in $E_{Q_0}$. 
\item $f_t\in \text{dom}(\mf)$ for all $-\epsilon<t<\epsilon$.
\end{enumerate}
\end{definition}
If (1) and (2) are satisfied, then by taking $\epsilon$ sufficiently small (3) is satisfied also. 

\begin{proposition}\label{propSmoothCurveVariation}
Let $f^* \in \text{dom}(\mf)$. Then
\begin{equation}\left.\frac{d}{dt}\mathcal{F}(f_t)\right|_{t=0}=0\end{equation}
for all curves $t\mapsto f_t$ satisfying (A1) with $f_0=f^*$ if and only if its Q-tensor, $Q^*$ is a critical point of $J$, and $f^*=f_{Q^*}$.
\end{proposition}
\begin{proof}
First we split $\mf(f_t)$ into parts so that its differentiability is clearer to see.
\begin{equation}
\begin{split}
\mathcal{F}(f_t)=& J(Q_t)+\mf(f_t)-\mf(f_{Q_t})\\
=& J(Q_t)+\left(\int_{\s2}f_t(p)\ln f_t(p)-f_{Q_t}(p)\ln f_{Q_t}(p)\,dp\right)+\int_{\s2}u_t(p) \ln (Q_t p\cdot p-\eta)\,dp\\
=&  J(Q_t)+\left(\int_{\supp(u_t)}f_t(p)\ln f_t(p)-f_{Q_t}(p)\ln f_{Q_t}(p)\,dp\right)+\int_{\supp(u_t)}u_t(p)\ln (Q_t p\cdot p-\eta)\,dp
\end{split}
\end{equation}
Now the derivative of this expression can be taken, where the support condition on $u_t$ removes any issues about non-differentiability of the terms involving logarithms at zero, and the sufficient differentiability of $f_{Q_t}(p)$ on its domain comes from \Cref{lemmaQToFC1}. 
\begin{equation}
\begin{split}
&\frac{d}{dt}\mf(f_t)\\
=&\frac{\partial J}{\partial Q}(Q_t)\cdot \frac{dQ_t}{dt}+\int_{\supp(u_t)} \left(\ln (f_t)-\ln (f_{Q_t})\right)\frac{\partial f_t(p)}{\partial Q}\frac{dQ_t}{dt}\,dp + \int_{\supp(u_t)}\ln (f_t(p))u(p)\,dp\\
&-\int_{\supp(u_t)}\frac{u_t(p)}{Q_tp\cdot p-\eta}\left(p\otimes p-\frac{1}{3}I\right)\, -u(p)\ln (Q_tp\cdot p-\eta)\,dp\\
=& \left(\frac{\partial J}{\partial Q}(Q_t)+\int_{\supp(u_t)}\left(\ln(f_t(p))-\ln(f_{Q_t}(p))\right)\frac{\partial f_{Q_t}(p)}{\partial Q}-\frac{u_t(p)}{Q_tp\cdot p-\eta}\left(p\otimes p-\frac{1}{3}I\right)\,dp\right)\cdot \frac{dQ_t}{dt}\\
&+\int_{\supp(u_t)}\left(\ln f_t(p)-\ln (Q_tp\cdot p-\eta)\right)u(p)\,dp
\end{split}
\end{equation}
If the curve goes through $f_{Q^*}$, where $Q^*$ is a critical point of $J$, then at $t=0$ $u_0=0$, $\supp(u_0)=\emptyset$ and all terms vanish, so that $\left.\frac{d}{dt}\mathcal{F}(f_t)\right|_{t=0}=0$. 

Conversely, assume that this vanishes at $t=0$ for all such curves satisfying (A1). Considering $Q$ constant in $t$, this implies
\begin{equation}\int_{\supp(u_0)}\ln \frac{f^*(p)}{Q^*p\cdot p-\eta} u(p)=0\end{equation}
In particular, since $u \in V$, 
\begin{equation}\ln\left(\frac{f_0}{Q_0p\cdot p-\eta}\right) = \alpha + \lambda p\cdot p\end{equation}
by Hahn-Banach, so that $f^*(p)=\frac{1}{z}\exp(\lambda p\cdot p)(Q^*p\cdot p-\eta)$ on $\supp{f^*}=E_{Q^*}$. By uniqueness of this solution, this implies that $\lambda=\Lambda(Q^*)$ and $z=Z$. In particular, $f^*=f_{Q^*}$, and $u_0=0$. Substituting this back into the expression for when $Q$ is not constant in $t$ gives that 
\begin{equation}
\begin{split}
\left.\frac{\partial J}{\partial Q}(Q_0)\cdot \frac{dQ_t}{dt}\right|_{t=0}=0,
\end{split}
\end{equation}
which then implies that $\frac{\partial J}{\partial Q}(Q_0)=0$, so that $Q_0$ is a critical point of $J$. 

\end{proof}

In Onsager-type models we can avoid difficulties in differentiating logarithms at zero if we restrict ourselves only to probability distributions bounded away from zero. A heuristic interpretation of the previous result is that if we look at the restricted set of probability distributions $f \in \ps2$ with $\supp(f)\subset E_Q$ so that $f(p)=f_Q(p)$ for $p$ near $\partial E_Q$, then we can avoid analogous differentiability issues in this model. 

\section{Related models}
\label{secRelatedModels}
\subsection{The inclusion of attractive, thermally dependent, interactions}
\label{secThermal}
We now consider an adjustment of the free-energy density accounting for attractive interactions, while the model had previously only considered repulsive steric interactions. As in Maier-Saupe, we consider attractive interactions dependent on temperature and the Q-tensor of our orientation distribution function. We take the new free energy density as
\begin{equation}k_BT\rho_0 \int_{\mathbb{S}^2}f(p)\ln f(p)-f(p)\ln \left(1-\rho_0\left(c-\frac{3}{2}dQp\cdot p\right)\right)\,dp - \rho_0^2U(a+b|Q|^2),\end{equation}
where $a,b,U$ are material parameters related to the anisotropy of the polarisability of the molecules and $b,U>0$. Dividing through by $\rho_0kT$, which leave the minimisation unchanged, this can be non-dimensionalised again as 
\begin{equation}
\begin{split}
&\int_{\mathbb{S}^2}f(p)\ln f(p)-f(p) \ln \left(Qp\cdot p-\frac{\rho_0c-1}{\rho_0d}\right)-\frac{\rho_0}{k_BT}\left(bU\right)|Q|^2+\frac{\rho_0Ua}{kT}-\ln(\rho_0d)\\
=&\int_{\mathbb{S}^2}f(p)\ln f(p)-f(p) \ln \left(Qp\cdot p-\eta\right)-\frac{1}{2\tau}|Q|^2 +C_1
\end{split}
\end{equation}
where $C_1$ is irrelevant to the minimisation and $\tau=\frac{2k_BT}{\rho_0bU}>0$. Since this is a $Q$-dependent perturbation of the original minimisation problem, this can easily be treated using the methodology given previously.
\begin{proposition}
All global and $L^p$-local minimisers of $\mathcal{F}_\tau$ are in one-to-one correspondence with global and local minimisers of $J_\tau$, where $J_\tau :\text{dom}(J)\to\mathbb{R}$ by 
\begin{equation}J_\tau(Q)=J(Q)-\frac{1}{2\tau}|Q|^2.\end{equation}
\end{proposition}
\begin{proof}
This follows from \Cref{remarkGeneralMinimisation}, noting that 
\begin{equation}\min\limits_{f \in A(Q)}\mathcal{F}_\tau(f)=\min\limits_{f \in A(Q)}\mf(f)-\frac{1}{2\tau}|Q|^2.\end{equation}
\end{proof}

\begin{proposition}\label{propBifurcation}
Let $\eta<0$. Then $Q=0$ is a critical point of $J_\tau$ for all $\tau>0$. Let $\tau_c=\frac{15}{2}\left(1+\frac{2}{15\eta}\right)^{-2}$. Then if $\tau<\tau_c$ $Q=0$ is a local minimum of $J_\tau$, and if $\tau>\tau_c$ then $Q$ is a local maximum of $J_\tau$.
\end{proposition}
\begin{proof}
The first derivative of $J_\tau$ is easily given by 
\begin{equation}\frac{\partial J_\tau}{\partial Q}(0)=\frac{\partial J}{\partial Q}(0)-\frac{1}{\tau}0,\end{equation}
at which point the critical point condition is satisfied, since $Q=0$ is always a local minimum of $J$. The second derivative at $0$, using \Cref{propSecondDerivativeJ}, is then given by 
\begin{equation}
\begin{split}
\frac{\partial^2J_\tau}{\partial Q^2}(0)=\left(\frac{15}{2}\left(1+\frac{2}{15\eta}\right)^2-\frac{1}{\tau}\right)\text{Id}.
\end{split}
\end{equation}
Therefore the second derivative is positive if $\tau <\frac{15}{2}\left(1+\frac{2}{15\eta}\right)^{-2}$, and is negative if $\tau>\frac{15}{2}\left(1+\frac{2}{15\eta}\right)^{-2}$.
\end{proof}

All of the results in this subsection can be summarised in the following theorem. 

\begin{theorem}\label{theoremMinsEquivThermal}
If $f^* \in \ps2$ is an $L^p$-local minimiser for any $p \in [1,\infty]$, then $f^*$ is an $L^q$-local minimiser for all $q \in [1,\infty]$. All $L^p$-local minimisers $f^*$ satisfy the Euler-Lagrange equation 
\begin{equation}
\begin{split}
f^*(p)=&\frac{1}{Z}\exp(\lambda^* p\cdot p)\max(Q^*p\cdot p-\eta,0),\\
Q^*=& \int_{\mathbb{S}^2}\left(p\otimes p-\frac{1}{3}I\right)f^*(p)\,dp,\\
\lambda^* =& \int_{\mathbb{S}^2}\frac{f^*(p)}{Q^*p\cdot p-\eta}\left(p\otimes p-\frac{1}{3}I\right)\,dp +\frac{1}{\tau}Q^*.
\end{split}
\end{equation}
$Q^*,\lambda^* $ are also saddle points of $F:\text{dom}(J)\times\sm3 \to\mathbb{R}$ given by 
\begin{equation}F(Q,\lambda)=Q\cdot \lambda -\frac{1}{2\tau}|Q|^2-\ln\left(\int_{\mathbb{S}^2}\exp(\lambda p\cdot p)\max(Qp\cdot p-\eta,0)\,dp\right).\end{equation}
Finally, the isotropic state $f_u=\frac{1}{4\pi}$ is an $L^p$-local minimiser if $\tau<\tau_c=\frac{15}{2}\left(1+\frac{2}{15\eta}\right)^{-2}$, and is not a local minimiser if $\tau>\tau_c$.
\end{theorem}

It is also possible to produce a statement similar to \Cref{propLimitingF} for the zero-temperature limit. 

\begin{proposition}
Let $\eta<\frac{2}{3}$, $\tau_j \nearrow 0$, and let $f_j \in \ps2$ be corresponding minimisers of $\mf_{\tau_j}$. Then there exists rotations $R_j \in \text{SO}(3)$ such that $[R_jf_j] \wsl \frac{1}{2}\left(\delta_{e_1}+\delta_{e_1}\right)$ in $C(\mathbb{S}^2)^*$.
\end{proposition}
\begin{proof}
Following the proof of \Cref{propLimitingF}, it suffices to show that the corresponding Q-tensors $Q_j$ satisfy $|Q_j|^2 \to \frac{2}{3}$. For the sake of contradiction assume otherwise, so that there exists some $0<R<\frac{2}{3}$ and a subsequence (not relabelled) so that $|Q_j|^2\leq R$ for all $j$. Let $Q=\frac{3(R+\epsilon)}{2}\left(e\otimes e-\frac{1}{3}I\right)$ with $\epsilon>0$ sufficiently small so that $R+\epsilon<\frac{2}{3}$, and let $f=f_Q$. Then 
\begin{equation}
\begin{split}
\mf_{\tau_j}(f_j)-\mf_{\tau_j}(f)=&J_{\tau_j}(Q_j)-J_{\tau_j}(Q)\\
\geq & \left(J(0)-\frac{1}{\tau_j}R^2\right)-\left(J(Q)-\frac{1}{\tau_j}(R+\epsilon)^2\right)\\
=& \left(J(0)-J(Q)\right)+\frac{\epsilon R+\epsilon^2}{\tau_j}.
\end{split}
\end{equation}
Taking $j \to \infty$ gives $\tau_j \to 0$ and the right-hand side blows up to $+\infty$. Therefore for sufficiently large $j$, $\mf_{\tau_j}(f_j)>\mf_{\tau_j}(f)$, contradicting that $f_j$ was a global minimiser. Therefore if $f_j$ is the sequence of minimisers, $\frac{2}{3}\geq\limsup\limits_{j \to \infty}|Q_j|^2\geq\liminf\limits_{j \to \infty}|Q_j|^2> R$. Since $R<\frac{2}{3}$ was arbitrary, $\lim\limits_{j\to\infty}|Q_j|^2=\frac{2}{3}$. 
\end{proof}

One immediate difference compared to the Maier-Saupe model is that for fixed temperature, by increasing the concentration it is possible to undergo a local isotropic-nematic-isotropic phase transition. In the original units, $b,V,U,c,d$ will generally be fixed and taken to be independent of the temperature and concentration, so it is possible to produce a local stability phase portrait for the isotropic state, changing only temperature and concentration by fixing these values (see \Cref{figureIsotropicStability}). In particular, as $\rho_0\to \frac{1}{c}$, the limiting concentration for the isotropic state to have finite energy, the transition temperature approaches zero. This can be seen since
\begin{equation}
\begin{split}
\tau_c=&\frac{15}{2}\left(1+\frac{2}{15\eta}\right)^{-2}=\frac{15}{2}\left(\frac{15\eta}{15\eta+2}\right)^{2},
\end{split}
\end{equation}
so as $\rho_0 \to \frac{1}{c}$, $\eta \to 0$, so $\tau_c \to 0$, and $T_c=\tau_c\frac{ \rho_0bU}{2} \to 0$. This analysis only accounts for local stability, and the numerical analysis in \cite{zheng2016density} would suggest that in all the region where $\eta>-\frac{2}{15}$, the isotropic state is globally unstable in the absence of attractive interactions, which since $J_\tau(0)=J(0)$ and $J_\tau(Q)<J(Q)$ for $Q\neq 0$, this then suggests that the isotropic state is globally unstable in this concentration regime when temperature effects are considered also. In particular, the concentration regime $-\frac{2}{15}<\eta<0$ where the isotropic state regains local stability would be difficult to observe in reality.

\begin{figure}[H]
\centering
        \includegraphics[width=0.45\textwidth]{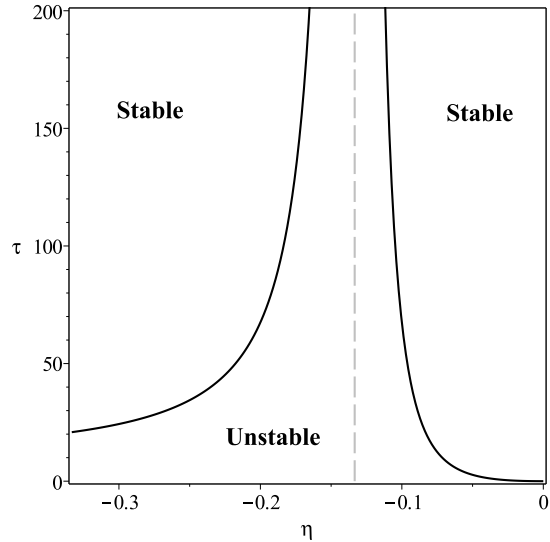}
    \caption{Local stability regions in terms of $\eta,\tau$ with the asymptote corresponding to $\eta=-\frac{2}{15}$}
    \label{figureIsotropicStability}
\end{figure}

\begin{remark}
Similarly to \Cref{remarkVdW}, we can obtain an equation of state with the inclusion of thermally attractive interactions. Returning to original units and using the reduction to $J$ again, we have
\begin{equation}
\begin{split}
\frac{1}{k_BT}P=&\left(-1+\rho\frac{\partial}{\partial \rho} \right)\left(\rho J\left(Q,\frac{2(\rho_0c-1)}{3\rho_0d}\right)-\frac{\rho^2U}{k_BT}(a+b|Q|^2)\right)\\
=&\rho\int_{\s2}\frac{f_Q(p)}{1-\rho (c-\frac{3}{2}dQp\cdot p)}\,dp-\frac{\rho^2U}{k_BT}(a+b|Q|^2).
\end{split}
\end{equation}
This can then be rearranged to be given in a more classical form, 

\begin{equation}
\left(P+\rho^2 U\left(a+b|Q|^2\right)\right)\left(\rho\int_{\s2}\frac{f_Q(p)}{1-\rho(c-dQp\cdot p)}\,dp\right)^{-1}=k_BT.
\end{equation}
In this case the average excluded volume term is given in terms of the harmonic mean rather than the arithmetic mean, although in the sphere limit where $d=0$, the two means are equal. 
\end{remark}

\subsection{Uniaxial systems}
\label{subsecUniaxial}
In order to reduce the complexity of the problem it is often assumed that nematic systems have axial symmetry about a fixed unit vector, removing a degree of freedom. Studies in Onsager models  suggest that in certain situations at least it can be energetically favourable for nematic systems to form such uniaxial systems \cite{fatkullin2005critical,vollmer2015critical}. As such it is frequently an assumption made in the modelling process, for example being a key assumption of the Oseen-Frank theory \cite{frank1958liquid}. In particular, this gives corresponding Q-tensors of the form $Q=S\left(e_1\otimes e_1-\frac{1}{3}I\right)$ with $S\in \left(-\frac{1}{2},1\right)$. We say that $Q$ is oblate if $S<0$ and prolate if $S>0$. The numerical studies in \cite{zheng2016density} invoke a uniaxial ansatz on solutions. It is possible to extend the previous results of this work to the case where systems are constrained to be uniaxial. Even if global/local minimisers of the full biaxial model are not uniaxial, the following results show the existence of prolate and oblate uniaxial critical points of the full biaxial problem in certain concentration regimes, with the precise statements given in \Cref{corollaryUniaxial1,corollaryUniaxial2}. First, we define the set of uniaxial probability distributions,
\begin{equation}\mathcal{P}_U(\s2)=\{f \in \ps2: \exists \tilde{f}:[-1,1]\to\mathbb{R},e\in \s2 \text{ such that } f(p)=\tilde{f}(p\cdot e)\},\end{equation}
and take the normalisation on the second Legendre polynomial 
\begin{equation}P_2(x)=\frac{1}{2}(3x^2-1).\end{equation}
With fixed eigenbasis, and $S \in \left(-\frac{1}{2},1\right)$ we take $Q_S=S\left(e_1\otimes e_1-\frac{1}{3}I\right)$.

\begin{proposition}\label{propMinsEquivUniaxial}
If $f\in \pus2$ is an $L^p$-local minimiser of $\mathcal{F}|_{\pus2}$, then there exists $\hat{f} \in \mathcal{P}([-1,1])$ and $n \in \mathbb{S}^2$ so that $\frac{1}{2\pi}\hat{f}(p\cdot n)=f(p)$ and 
\begin{equation}
\begin{split}
\hat{f}(x)=& \frac{1}{Z}\exp(lP_2(x))\max(SP_2(x)-\eta,0),\\
S=&\int_{-1}^1 P_2(x)\hat{f}(x)\,dx,\\
l=& \frac{2}{3}\int_{-1}^1 P_2(x)\frac{\hat{f}(x)}{SP_2(x)-\eta}\,dx.
\end{split}
\end{equation}
\end{proposition}
\begin{proof}
If $f$ is uniaxial about $n \in \mathbb{S}^2$, then its corresponding Q-tensor can be written as $Q=S\left(n\otimes n-\frac{1}{3}I\right)$, with $S=\int_{-1}^1P_2(x)\hat{f}(x)\,dx$. Therefore returning to the macroscopic function $J$, any local minimum of the uniaxially constrained Q-tensor problem can be described by its order parameter $S$, so by fixing the axis of symmetry,
\begin{equation}
\begin{split}
0=&\frac{d}{dS}J\left(S\left(n\otimes n-\frac{1}{3}I\right)\right)\\
=& \frac{\partial J}{\partial Q}\left(S\left(n\otimes n-\frac{1}{3}I\right)\right)\cdot \frac{\partial}{\partial S} S\left(n\otimes n-\frac{1}{3}I\right)\\
=& \left(\Lambda(Q_S) -\frac{1}{Z}\int_{\mathbb{S}^2}\frac{\exp(\Lambda(Q_S) p\cdot p)}{SP_2(p\cdot n)-\eta}\left(p\otimes p-\frac{1}{3}I\right)\,dp\right)\cdot \left(n\otimes n-\frac{1}{3}I\right).
\end{split}
\end{equation}
Now recall from \Cref{propQLambdaEigenvectors} that since $\Lambda(Q_S)$ shares an eigenbasis with $Q_S$, there exists a scalar $l$ so that $\Lambda(Q_S)=l\left(n\otimes n-\frac{1}{3}I\right)$. Therefore this can be written as 
\begin{equation}
\begin{split}
0=& \frac{2l}{3}-\frac{1}{Z}\int_{\mathbb{S}^2}\frac{\exp(lP_2(p\cdot n))}{SP_2(p\cdot n)-\eta}P_2(p\cdot n)\,dp\\
=&\frac{2l}{3}-\int_{-1}^1\frac{\hat{f}(x)}{SP_2(x)-\eta}P_2(x)\,dx.
\end{split}
\end{equation}
The dual optimality condition is similarly rephrased in terms of $\hat{f}$ as 
\begin{equation}S=\frac{1}{Z}\int_{\mathbb{S}^2}\exp(\Lambda(Q_S) p\cdot p)\max(SP_2(p\cdot n)-\eta,0)f(p)\,dp=\int_{-1}^1\exp(lP_2(x))\max(SP_2(x)-\eta,0)\hat{f}(x)\,dx.\end{equation}
The equivalence of local minimisation comes from the argument in \Cref{remarkGeneralMinimisation}.
\end{proof}

\begin{proposition}
Assume that $\hat{f}$ satisfies the uniaxial Euler-Lagrange equation. Then given $n \in \s2$, $f \in \ps2$ defined by $f(p)=\frac{1}{2\pi}\hat{f}(p\cdot n)$ satisfies the full biaxial Euler-Lagrange equation.
\end{proposition}
\begin{proof}
The dual optimality condition automatically holds by \Cref{propQLambdaEigenvectors}, so it only remains to show that the derivative of $J$ is zero at the corresponding Q-tensor. By writing the derivative condition in the uniaxial case as 
\begin{equation}\frac{\partial J}{\partial Q}(Q)\cdot \left(n \otimes n-\frac{1}{3}I\right)=0,\end{equation}
it suffices to show that $\frac{\partial J}{\partial Q}(Q)$ is parallel to $\left(n\otimes n-\frac{1}{3}I\right)$, or equivalently that $\frac{\partial J}{\partial Q}$ is uniaxial with distinguished direction $n$. Recalling that the derivative is given by 
\begin{equation}\frac{\partial J}{\partial Q}(Q)=\Lambda(Q) -\frac{1}{Z}\int_{E_Q}\exp(\Lambda(Q) p \cdot p)\left(p\otimes p-\frac{1}{3}I\right)\,dp,\end{equation}
and that if $Q$ is uniaxial then so is $\Lambda(Q)$ with the same distinguished direction, the result immediately follows. 
\end{proof}

\begin{remark}
This tells us that if $Q$ is uniaxial and a critical point of the constrained problem, then $Q$ is a critical point of the unconstrained problem. However, this does not say that uniaxially constrained (local) minimisers are necessarily unconstrained (local) minimisers. Nor does it imply that all critical points are uniaxial. We can however deduce the existence of certain uniaxial critical points of the biaxial system.
\end{remark}

\begin{corollary}
\label{corollaryUniaxial1}
For $0\leq \eta < \frac{2}{3}$, there exists a prolate uniaxial critical point of $J$. Furthermore, there exists rotations $R_j\in\text{SO}(3)$ so that as $\eta_j \to \frac{2}{3}$ from below, $[R_jf_j] \wsl \frac{1}{2}\left(\delta_{e_1}+\delta_{-e_1}\right)$ in $C(\s2)^*$.
\end{corollary}
\begin{proof}
Consider the map $S\mapsto J\left(S\left(e_1 \otimes e_1-\frac{1}{3}I\right)\right)$ for $S \in \left(\sqrt{\frac{3\eta}{2}},1\right)$. This is a continuous map on an open set, which blows up at the boundary by \Cref{propJBlowUp}. Therefore it admits a local minimum, and the derivative must vanish there. Therefore the corresponding prolate uniaxial Q-tensor is a critical point of $J$. The convergence result follows by the same argument as \Cref{propLimitingF}
\end{proof}

\begin{corollary}
\label{corollaryUniaxial2}
If $0\leq \eta < \frac{1}{6}$. There exists an oblate uniaxial critical point of $J$. Furthermore, there exists rotations $R_j\in\text{SO}(3)$ so that as $\eta_j \to \frac{1}{6}$ from below, $[R_jf_j] \wsl \frac{1}{2\pi}\mathcal{H}^1\mesr\{p \in \s2 :e\cdot p=0\}$ in $C(\s2)^*$.
\end{corollary}
\begin{proof}
The proof is similar to the previous result, by taking the map $S\mapsto J\left(S\left(e \otimes e-\frac{1}{3}I\right)\right)$ restricted to $S \in \left(-\frac{1}{2},-\sqrt{\frac{3\eta}{2}}\right)$. By the same argument this produces an oblate uniaxial critical point of $J$, and the corresponding probability distribution satisfies the Euler-Lagrange equation. 

Now let $\eta_j \to \frac{1}{6}$ from below, let $S_j$ be the corresponding local minimiser and $f_j$ be the probability distribution corresponding to the local minimiser $S_j\left(e_j\otimes e_j-\frac{1}{3}I\right)$ for some $e_j\in\s2$. The proof follows a similar argument to \Cref{propLimitingF}. We first let $R_j$ be any rotation such that $R_j^Te_j=e$. If $\hat{f}_j=[R_jf_j]$, then $\hat{f}_j$ has Q-tensor $S_j\left(e\otimes e-\frac{1}{3}I\right) \to -\frac{1}{2}\left(e\otimes e-\frac{1}{3}I\right)$. We have that $|Q_j|^2 \to \frac{1}{6}$, and in particular $S_j^2 \to \frac{1}{6}$. This implies that given $\epsilon>0$, for sufficiently large $j$, $\text{supp}(\hat{f}_j)\subset \{ p\in\s2 : |e\cdot p|<\epsilon\}=U_\epsilon$. Let $g \in C(\s2)$, and let $\tilde{g}:[0,\pi]\times[0,2\pi]$ denote its representation in spherical coordinates, with $\theta=0$ corresponding to $e$. Similarly, let $\tilde{f}_j:[0,\pi]\to\mathbb{R}$ be the representation of $\hat{f}_j$ in spherical coordinates, which by axial symmetry is independent of $\phi$. We can then write 
\begin{equation}
\begin{split}
\int_{\s2}\hat{f}_j(p)g(p)\,dp=& \int_0^\pi \int_0^{2\pi}\sin(\theta)\tilde{f}_j(\theta)\tilde{g}(\theta,\phi)\,d\phi\,d\theta\\
=& \int_0^\pi \sin(\theta)\tilde{f}_j(\theta)\left(\int_0^{2\pi}\tilde{g}(\theta,\phi)\,d\phi\right)\,d\theta\\
=& \int_{\s2}\hat{f}_j(p)\frac{g_s(p)}{2\pi}\,dp,
\end{split}
\end{equation}
where $g_s(p)$ is the axially symmetric function defined in spherical coorindates by \begin{equation}\tilde{g_s}(\theta)=\int_0^{2\pi} \tilde{g}(\theta,\phi)\,d\phi.\end{equation}
Since $g$ is continuous, this implies that $g_s$ is continuous also. Furthermore, $\text{supp}(f_j)\subset U_\epsilon$ for all large $j$, and for each $\epsilon>0$, $U_\epsilon$ is a neighbourhood of $U_0=\{p\cdot e=0\}$, and $\bigcap\limits_{\epsilon>0}U_\epsilon=U_0$. Therefore for large $j$,
\begin{equation}
\begin{split}
\int_{\s2}\hat{f}_j(p)g(p)\,dp=& \int_{\s2}f_j(p)\frac{g_s(p)}{2\pi}\,dp\\
=& \int_{U_\epsilon} \hat{f}_j (p) \frac{g_s(p)}{2\pi}\,dp\\
\in & \left[ \min\limits_{p \in U_\epsilon} \frac{g_s(p)}{2\pi},\max\limits_{p \in U_\epsilon} \frac{g_s(p)}{2\pi}\right].
\end{split}
\end{equation}
As in \Cref{propLimitingF}, a squeeze argument then gives that 
\begin{equation}\lim\limits_{j \to\infty} \int_{\s2}\hat{f}_j(p)g(p)\,dp  \in \left[ \min\limits_{p \in U_0} \frac{g_s(p)}{2\pi},\max\limits_{p \in U_0} \frac{g_s(p)}{2\pi}\right]=\left\{\frac{1}{2\pi}g_s(e_2)\right\},\end{equation}
where $e_2$ can be any vector orthogonal to $e$, due to the axial symmetry of $g_s$. Substituting this back into the definition of $g_s$ then gives that 
\begin{equation}\lim\limits_{j \to +\infty} \int_{\s2}[R_jf_j](p)g(p)\,dp=\frac{1}{2\pi}\tilde{g_s}(0)=\int_{0}^{2\pi} \tilde{g}(0,\phi)\,d\phi=\frac{1}{2\pi}\int_{\{p \cdot e_1=0\}}g(p)\,d\mathcal{H}^1(p).\end{equation}
\end{proof}

\begin{corollary}\label{corollaryUniaxial3}
There exists $\eta_0<0$ so that for all $\eta_0<\eta<0$, there exists two uniaxial critical points of $J$, one oblate and one prolate, denoted $Q^+_\epsilon$ and $Q^-_\epsilon$ respectively, such that $Q^\pm_\eta \to 0$ as $\eta \nearrow 0$. Furthermore, these are local maxima of the uniaxially constrained problem.
\end{corollary}
\begin{proof}
Only the existence of the prolate branch will be shown, with the oblate branch proof being identical. We first show that there exists a branch of local maximisers as stated, from which the result follows. Let $\eta_j \nearrow 0$ with $\eta_j>-\frac{2}{15}$ for all $j \in\mathbb{N}$. In particular, the isotropic state is a strict local minimiser for all $\eta_j$. Define $S_j=\inf\{ S \in (0,1): J(S,\eta_j)<J(0,\eta_j)\}$. The infima may be over an empty set, but for $j$ sufficiently large we can show this set is non-empty and that as $j \to +\infty$, $S_j\to 0$. Let $\epsilon >0$ be small. Since $J$ is continuous, $J\left(\epsilon\frac{1}{2},\eta_j\right)-J(0,\eta_j)\to -\infty$. In particular, for large $j$, $J\left(\epsilon,\eta_j\right)<J(0,\eta_j)$, proving that the set is non-empty. Furthermore, since $\epsilon$ was arbitrary, $S_j \to 0$. Since the isotropic state is a strict local minimiser, and $J(S_j,\eta_j)=J(0,\eta_j)$, this implies that for large $j$, there exists a local maximum $S^+_{\eta_j}$ so that $0<S^+_{\eta_j}<S_j \to 0$. Since $S^+_{\eta_j}$ is a local maximum of a $C^1$ function it must be a critical point, and therefore there exists a corresponding critical point of the unconstrained biaxial problem. 
\end{proof}

\subsubsection{Summary and qualitative phase diagram}
\begin{figure}[h!]
\begin{minipage}{0.9\textwidth}
\begin{center}
\includegraphics[width=0.7\textwidth]{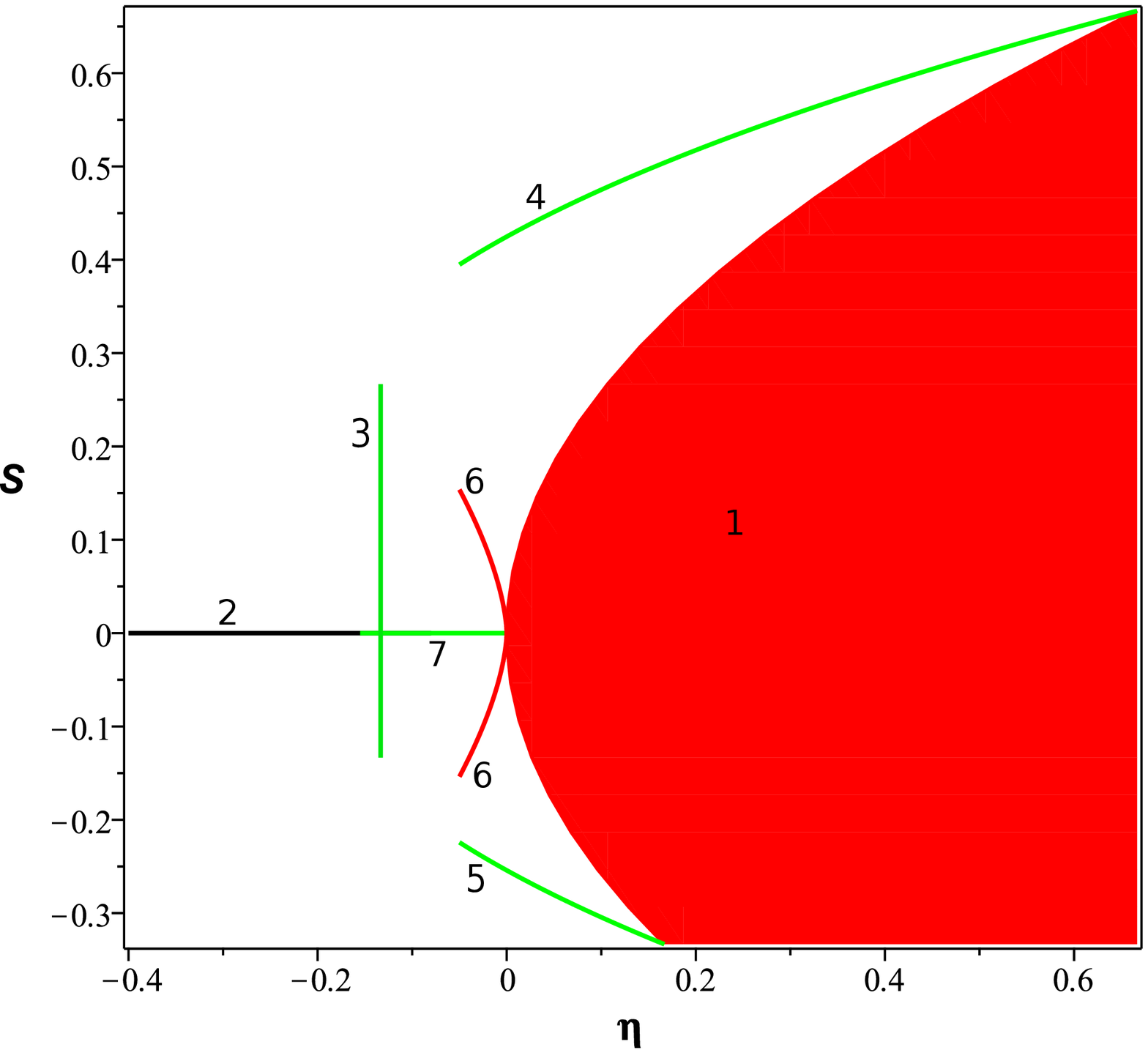}
\end{center}
\begin{multicols}{2}
\begin{enumerate}
\item Inadmissibility region from \Cref{propAdmissibleQ}.
\item Globally stable branch of the isotropic state from \Cref{propositionIsotropicStateGlobalMin}.
\item The non-strict local minimisers at $\eta=-\frac{2}{15}$ from \Cref{propEta215Drama}.
\item The prolate uniaxial branch from \Cref{corollaryUniaxial1}.
\item The oblate uniaxial branch from \Cref{corollaryUniaxial2}.
\item The uniaxial unstable branches that approach $0$ from \Cref{corollaryUniaxial3}.
\item The existence of some $\eta^c<0$ where the isotropic state ceases to be a global minimiser from \Cref{propIsotropicNotGlobal}. 
\end{enumerate}
\end{multicols}
\end{minipage}
\caption{The qualitative features of the phase diagram obtained analytically.}
\label{figQualitativePhaseDiagram}
\end{figure}

In the work of Zheng {\it et. al.} a phase diagram for the uniaxially constrained model was obtained numerically. We can compare this with a qualitative version of the phase diagram, obtained from the analytical results obtained in this work. In \Cref{figQualitativePhaseDiagram} we show the known branches of the phase diagram, labeled according to the results from which their existence was shown. This phase diagram also acts as a summary of the main results of the this work. Within the diagram green lines correspond to local minima, black lines to global minima, and red lines to unstable critical points. The large red region is the inadmissibility region, where $|Q|^2>\eta$. It is stressed to the reader that this is a caricature to show only the branches, and the quantitative values should be ignored. We can draw the conclusion that the numerical and analytical results are consistent, in as far as every branch of solutions in each phase diagram is found in the other. What remains to be understood analytically is the quantitative values of these branches, which branches correspond to global minima, and the bifurcations of the branches.

\section*{Acknowledgements}
The author would like to thank John M. Ball, Peter Palffy-Muhoray and Xiaoyu Zheng and Epifanio Virga for insightful and motivating discussions that have benefited this work. The research leading to these results has received funding from the European Research Council under the European Union's Seventh Framework Programme (FP7/2007-2013) / ERC grant agreement n$^{\circ}$ 291053.
\bibliographystyle{acm}
\bibliography{bibl}

\appendix

\section{The second derivative of $J$ at $0$}
\begin{proposition}\label{propFourthOrder}
Let $M$ be the fourth-order tensor defined by
\begin{equation}M=\frac{1}{4\pi}\int_{\s2}p\otimes p\otimes p\otimes p\,dp.\end{equation}
Acting as a linear operator from $\sm3$ to itself, given by 
\begin{equation}M(A)_{ij}=\sum\limits_{\alpha,\beta=1}^3M_{\alpha\beta i j}A_{\alpha\beta},\end{equation}
can be simplified as $M(A)=\frac{2}{15}A$.
\end{proposition}
\begin{proof}
If $A$ is a symmetric matrix (not necessarily traceless) and $R \in \text{SO}(3)$, then $M(RAR^T)=RM(A)R^T$, which is seen by using a change of variables $q=Rp$ in the integrand. By \cite{smith1971isotropic}, this implies that there are continuous scalar functions $g_0,g_1,g_2$ so that $M(A)=\sum\limits_{i=0}^3g_i(A)A^i$. Viewed only as a map from $\sm3$ to itself, this simplifies as $M(A)=g_2(A)\left(A^2-\frac{|A|^2}{3}I\right)+g_1(A)$. Since $M$ is a linear map, we have that for any $r\in\mathbb{R}\setminus \{0\}$,
\begin{equation}
\begin{split}
M(A)=&\frac{1}{r}M(r A)\\
=&\frac{1}{r} \left(r^2 g_2(r A)\left(A^2-\frac{|A|^2}{3}I\right)+r g_1(r A)A\right)\\
=& r g_2(r A)\left(A^2-\frac{|A|^2}{3}I\right)+g_1(r A)A\\
\Rightarrow M(A)=&\lim\limits_{r \to 0} r g_2(r A)\left(A^2-\frac{|A|^2}{3}I\right)+g_1(r A)A\\
=& g_1(0)A,
\end{split}
\end{equation}
using that $g_1,g_2$ are continuous. We have therefore established that $M$ is simply a multiple of the identity. To establish this constant, let $A=\left(e\otimes e -\frac{1}{3}I\right)$ for some $ e \in \s2$. Then the multiple can be established as 
\begin{equation}
\begin{split}
g_1(0)=& \frac{M(A)\cdot A}{A\cdot A}\\
=& \frac{3}{2}\frac{1}{4\pi}\int_{\s2}\left((p\cdot e)^2-\frac{1}{3}\right)\,dp\\
=& \frac{3}{2}\frac{1}{2}\int_{-1}^1 \left(x^2-\frac{1}{3}\right)^2\,dx\\
=&\frac{3}{2}\frac{4}{45}=\frac{2}{15}.
\end{split}
\end{equation}
\end{proof}

\begin{proposition}\label{propSecondDerivativeJ}
The second derivative of $J$ at $Q=0$ is given by 
\begin{equation}\frac{\partial^2 J}{\partial Q^2}(0)=\frac{15}{2}\left(1+\frac{2}{15\eta}\right)^2\text{Id},\end{equation}
where $\text{Id}$ is the identity operator on $\sm3$.
\end{proposition}
\begin{proof}
First recall that $J$ can be written as $J(Q)=\max_{\lambda\in\sm3} F(Q,\lambda)=F(Q,\Lambda(Q))$ with 
\begin{equation}F(Q,\lambda) =Q\cdot \lambda - \ln \left(\int_{\mathbb{S}^2}\exp(\lambda p \cdot p)(Qp\cdot p-\eta)\,dp\right).\end{equation}
Thus the derivative of $J$ can be written as 
\begin{equation}
\begin{split}
\frac{\partial J}{\partial Q_{ij}} =& \frac{\partial F}{\partial Q_{ij}} + \frac{\partial F}{\partial \lambda_{\alpha\beta}}\frac{\partial \Lambda_{\alpha\beta}}{\partial Q_{ij}}\\
=& \frac{\partial F}{\partial Q_{ij}}
\end{split}
\end{equation}
due to the optimality condition of $\Lambda(Q)$. Similarly the second derivative can be given as 
\begin{equation}
\begin{split}
\frac{\partial^2F}{\partial Q_{ij}\partial Q_{kl}}=& \frac{\partial^2F}{\partial Q_{ij}\partial Q_{kl}}+ \frac{\partial^2F}{\partial \lambda_{\alpha\beta}\partial Q_{ij}}\frac{\partial \Lambda_{\alpha\beta}}{\partial Q_{kl}}.
\end{split}
\end{equation}
Furthermore the derivative of $\Lambda$ with respect to $Q$ can be found as 
\begin{equation}
\begin{split}
0=&\frac{\partial F}{\partial \lambda_{ij}}(Q,\Lambda(Q))\\
0=& \frac{\partial^2 F}{\partial \lambda_{ij}\partial Q_{kl}}+\frac{\partial^2F}{\partial \lambda_{ij}\partial \lambda_{\alpha\beta}}\frac{\partial\Lambda_{\alpha\beta}}{\partial Q_{kl}}\\
-\left(\frac{\partial^2F}{\partial\lambda^2}\right)^{-1}_{ij,\alpha\beta}\frac{\partial^2 F}{\partial \lambda_{\alpha\beta}\partial Q_{kl}}=&\frac{\partial \Lambda_{ij}}{\partial Q_{kl}}.
\end{split}
\end{equation}
By an argument similar to \cite{taylor2015maximum}, 
\begin{equation}\frac{\partial^2F}{\partial\lambda^2}(Q,\Lambda(Q))=Q\otimes Q-\int_{\mathbb{S}^2}f_Q(p)\left(p\otimes p-\frac{1}{3}I\right)^{\otimes 2}\,dp .\end{equation}
At the isotropic state, this reduces to 
\begin{equation}\frac{\partial ^2F}{\partial \lambda^2}(0,0)=-\frac{1}{4\pi}\int_{\mathbb{S}^2}\left(p\otimes p-\frac{1}{3}I\right)^{\otimes 2}\,dp.\end{equation}
Using \Cref{propFourthOrder}, viewing this as a linear operator on $\sm3$, and denoting $\text{Id}:\sm3\to\sm3$ the identity operator, this identified with 
\begin{equation}\frac{\partial^2F}{\partial\lambda^2}(0,0)=-\frac{2}{15}\text{Id}.\end{equation}
Next to find the mixed derivative 
\begin{equation}
\begin{split}
\frac{\partial^2F}{\partial \lambda_{ij}\partial Q_{kl}}=& \frac{\partial}{\partial Q_{kl}}\left(Q_{ij}-\frac{1}{Z}\int_{\mathbb{S}^2}\exp(\lambda p\cdot p)(Qp\cdot p-\eta)\left(p_ip_j-\frac{1}{3}\delta_{ij}\right)\,dp\right)\\
=&\delta_{ik}\delta_{jl}-\frac{\int_{\mathbb{S}^2}\exp(\lambda p\cdot p)\left(p_ip_j-\frac{1}{3}\delta_{ij}\right)\left(p_kp_l-\frac{1}{3}\right)\,dp}{Z}\\
&+\frac{\int_{\mathbb{S}^2}\exp(\lambda p\cdot p)(Qp\cdot p-\eta)\left(p_ip_j-\frac{1}{3}\delta_{ij}\right)\,dp\int_{\mathbb{S}^2}\exp(\lambda p\cdot p)(Qp\cdot p-\eta)\left(p_kp_l-\frac{1}{3}\delta_{kl}\right)\,dp}{Z^2}.
\end{split}
\end{equation}
Taking $Q=\lambda=0$, then 
\begin{equation}
\begin{split}
\frac{\partial^2F}{\partial \lambda_{ij}\partial Q_{kl}}(0,0)=&\delta_{ik}\delta_{jl}+\frac{1}{4\pi\eta}\int_{\mathbb{S}^2}\left(p_ip_j-\frac{1}{3}\delta_{ij}\right)\left(p_kp_l-\frac{1}{3}\delta_{kl}\right)\,dp\\
=& \left(1+\frac{2}{15\eta}\right)\text{Id}_{ijkl}.
\end{split}
\end{equation}
Finally, for the second derivative with respect to $Q$, 
\begin{equation}
\begin{split}
\frac{\partial^2 F}{\partial Q_{ij}\partial Q_{kl}}=& \frac{\partial}{\partial Q_{kl}}\left(\lambda_{ij}-\frac{1}{Z}\int_{\mathbb{S}^2}\exp(\lambda p \cdot p)\left(p_ip_j-\frac{1}{3}\delta_{ij}\right)\,dp\right)\\
=&\frac{1}{Z^2}\int_{\mathbb{S}^2}\exp(\lambda p \cdot p)\left(p_ip_j-\frac{1}{3}\delta_{ij}\right)\,dp\int_{\mathbb{S}^2}\exp(\lambda p \cdot p)\left(p_kp_l-\frac{1}{3}\delta_{kl}\right)\,dp,
\end{split}
\end{equation}
therefore evaluating at $Q=\lambda=0$ gives $\frac{\partial^2F}{\partial Q^2}(0,0)=0$. Now combining these results gives 
\begin{equation}
\begin{split}
\frac{\partial^2 J}{\partial Q_{ij}\partial Q_{kl}}(0)=& \frac{\partial^2F}{\partial Q_{ij}\partial Q_{kl}}(0,0)+ \frac{\partial^2F}{\partial \lambda_{\alpha\beta}\partial Q_{ij}}(0,0)\frac{\partial \lambda_{\alpha\beta}}{\partial Q_{kl}}(0,0)\\
=& \left(1+\frac{2}{15\eta}\right)\text{Id}_{\alpha\beta ij}\cdot \frac{15}{2}\text{Id}_{\alpha\beta\gamma\delta}\left(1+\frac{2}{15}\eta\right)\text{Id}_{\gamma\delta kl}\\
=&\frac{15}{2}\left(1+\frac{2}{15\eta}\right)^2\text{Id}_{ijkl}.
\end{split}
\end{equation}
\end{proof}

\end{document}